\documentclass[a4paper,12pt,english,reqno]{smfart}
\usepackage[latin1]{inputenc}
\usepackage{amsmath}
\usepackage{amsfonts}
\usepackage{amssymb}
\usepackage{graphicx}
\usepackage{graphics}
\usepackage{epsfig}
\usepackage{amscd}
\usepackage{color}
\usepackage{shadow}
\usepackage{a4wide}
\usepackage{here}
\usepackage[all]{xy}
\usepackage{latexsym}
\usepackage{hyperref,amsthm}
\usepackage{amsopn}
\usepackage{epic, eepic}
\usepackage{fancyhdr}
\newtheorem{th1}{Theorem}[section]

\newtheorem{lem}[th1]{ Lemma}

\author[H.Hani and M. Khenissi]{Houda Hani and Moez Khenissi}
\title[\textbf{Asymptotic behaviours}]{Asymptotic behaviours of solutions for finite difference analogue of the Chipot-Weissler equation}

\begin{document}
	\frontmatter
	\begin{abstract}
		This paper deals with nonlinear parabolic equation for which a local solution in time exists and then blows up in a finite time. We consider the Chipot-Weissler equation:
		\begin{equation*}
		u_{t}=u_{xx}+u^{p}-\left|u_{x}\right|^{q},\ \ x\in (-1,1);\ t>0, \ \ p>1 \text{ and } 1\leq q\leq \dfrac{2p}{p+1}.
		\end{equation*}
		We study the numerical approximation, we show that the numerical solution converges to the continuous one under some restriction on the initial data and the parameters $p$ and $q$. Moreover, we study the numerical blow up sets and we show that although the convergence of the numerical solution is guaranteed, the numerical blow up sets are sometimes different from that of the PDE. 
	\end{abstract}
	\frontmatter
	\maketitle
	\mainmatter\date{}
	\tableofcontents
	\section{Introduction}
	In this paper, we consider the nonlinear parabolic problem
	\begin{equation}
	\left\{
	\begin{array}{lll}
	u_{t}=u_{xx}+u^{p}-\left|u_{x}\right|^{q},\ \ \ x\in (-1,1),\ t>0,\\
	u(\pm 1,t)=0,\ \ t>0,\\
	u(x,0)=u_{0}(x),\ \ x\in(-1,1).	
	\end{array}
	\right.
	\label{exacte}
	\end{equation}
	Here $p>1,\ 1\leq q\leq \dfrac{2p}{p+1}$ and $u_{0}$ is a positive function which is compatible with the boundary condition. It is well known that for some initial data, this problem blows up in a finite time. Problem \eqref{exacte} was studied for the first time by Chipot and Weissler in \cite{chipotweissler}, since then, the phenomenon of blow up for different problems has been the issue of intensive study, see for example \cite{friedman},\cite{fujita},\cite{hayakawa},\cite{levine},\cite{souplet} and the references therein. There exists many theoretical studies on the question of the occurence of blow up, but from a numerical point of view, many interesting numerical questions for problem \eqref{exacte} are not treated.\\
	
	\noindent We define the blow-up set for problem \eqref{exacte} as:
	\begin{equation*}
	B(u)=\left\{ x\in [-1,1];\ \ \exists \  (x_{n},t_{n})\rightarrow (x,T^{*})\ \text{ such that } u(x_{n},t_{n})\rightarrow +\infty  \text{ as } n\rightarrow +\infty \right\}.
	\end{equation*}
	It is proved in \cite{chiblikfila} that the solution of \eqref{exacte} blows up only at the central point, that is:
	\begin{equation*}
	\exists \ \ T^{*}<+\infty  \text{ such that }  \lim_{t\rightarrow T^{*}}u(t,0)=+\infty \text{ but }\ \lim_{t\rightarrow T^{*}}u(t,x)<\infty \text{   when  }  x\neq 0.
	\end{equation*}
	In \cite{hani}, we have conctructed a finite difference scheme whose solution satisfies the same properties as the exact solution and moreover, we have proved that its solution blows up in a finite time. In this paper and for the same scheme, we show the convergence of the numerical solution to the continuous one under some restrictions on $p$ and $q$, and we study the asymptotic behaviour of the solution near its singularity. We prove that the numerical solution can blow up at more than one point, while a one point blow up is known to occur in the continuous problem. More precisely, we show that even if a difference solution blows up, its values remain bounded up to the moment of blow up except at the maximum point and its adjacent points, moreover, the number of blow up points depends, in a way, on the value of the parameter $q$.\\
	
	\noindent We recall the scheme studied in \cite{hani}, for $j=1,...,N_{n}$ and $n\geq 0$ we have
		\begin{equation}
		\left\{ 
		\begin{array}{lllll}
		\dfrac{u_{j}^{n+1}-u_{j}^{n}}{\tau_{n}}=\dfrac{u_{j+1}^{n+1}-2u_{j}^{n+1}+u_{j-1}^{n+1}}{h_{n}^{2}}+(u_{j}^{n})^{p}-\dfrac{1}{(2h_{n})^{q}}\left| u_{j+1}^{n}-u_{j-1}^{n}\right| ^{q-1}\left| u_{j+1}^{n+1}-u_{j-1}^{n+1}\right| ,\\
		u_{j}^{0}=u_{0}(x_{j}),\\
		u_{0}^{n}=u_{N_{n}+1}^{n}=0.\\
		\end{array}
		\right.
		\label{approchee}
		\end{equation}
	\noindent We denote by $U^{n}:=(u_{0}^{n},...,u_{N_{n}+1}^{n})^{t}$ the numerical solution of \eqref{approchee}, and $$\left\|U^{n}\right\|_{\infty}=\max\limits_{1\leq j\leq N_{n}}|u_{j}^{n}|$$ the $L^{\infty}$ norm of $U^{n}$.\\
	Here the notation $u_{j}^{n}$ is employed to denote the approximation of $u(x_{j},t^{n})$ for $x_{j}\in [-1,1]$ and $t^{n}\geq 0.$ Also, we fix other notations as follow:
	\begin{enumerate}
		\item $\tau: $ size parameter for the variable time mesh $\tau_{n}$.
		\item $h: $ size parameter for the variable space mesh $h_{n}$.	 
		\item $t^{n}$: $n$-th time step on $t>0$ determined as:	
		\begin{equation*}
		\left\{
		\begin{array}{lll}
		t_{0}=0\\
		t_{n}=t^{n-1}+\tau_{n-1}=\sum\limits_{k=1}^{n-1}{\tau_{k}}; \ n\geq 1.\\	
		\end{array}
		\right.\\
		\end{equation*}	
		\item $x_{j}$: $j$-th net point on $[-1,1]$ determined as:
		\begin{equation*}
		\left\{ 
		\begin{array}{lll}
		x_{0}=-1,\\
		x_{j}=x_{j-1}+h_{n},\  j\geq 1 \text{ and }\  n\geq 0,\\
		x_{N_{n}+1}=1.
		\end{array}
		\right.\\
	    \end{equation*}
		We suppose that a spatial net point $x_{m}$ coincides with the middle point $x=0.$
		\item $\tau_{n}: $ discrete time increment  of $n-$th step determined by
		\begin{equation*}
		\tau_{n}=\tau \min\left(1,\left\|U^{n}\right\|_{\infty}^{-p+1}\right).
		\end{equation*} 
		\item $h_{n}: $ discrete space increment of $n-$th step determined by $$h_{n}=
		\min\left(h,\left(2\left\|U^{n}\right\|_{\infty}^{-q+1}\right)^{\frac{1}{2-q}}\right).$$ 
		\item $N_{n}=\dfrac{1}{h_{n}}-1$ the number of subdivisions of the interval $[-1,1].$
		\item $m=\dfrac{N_{n}+1}{2}.$
	\end{enumerate}
	As in \cite{hani}, we suppose that the initial data $u_{0}$ satisfies the following conditions:\\
	(A1) $u_{0}$ is continuous, nonconstant and nonnegative in $[-1,1].$\\
	(A2) $u_{0}$ is spatially symmetric about $x=0.$\\
	(A3) $u_{0}$ is strictly monotone increasing in $[-1,0].$\\
	(A4) $u_{0}(-1)=u_{0}(1)=0.$\\
	(A5) $u_{0}$ is large in the sense that $\left\|u_{0}\right\|_{\infty}>>1.$\\
	
	This paper is organized as follows: In section 2, we state and prove the main results, that is, if $p=2$ and $q=1$ then the solution blows up at the maximum point and the points around it, but remains bounded at all of the rest points, while if $p>2$ and $q<\dfrac{2(p-1)}{p}$, then there is only a single point for the solution to blow up. In section 3, we prove the convergence of the numerical solution to the exact one. In section 4, we give an approximation of the blowing-up time. Finally, in section 5, we present some numerical simulations.
	\section{Main theorems}
	In this section, we study the asymptotic behaviour of the difference solution near the maximal point $x_{m}.$
	\begin{th1}
		Let $U^{n}$ be a solution of \eqref{approchee}, we suppose that $h<\dfrac{1}{1+\tau}$. For $p=2$ and $q=1$, we have 
		\begin{equation*}
		\lim_{n\rightarrow +\infty} u_{m-1}^{n}=\lim_{n\rightarrow +\infty} u_{m+1}^{n}=+\infty.
		\end{equation*}
	\end{th1}
	\begin{proof}
		For $j=m-1$, the equation of \eqref{approchee} can be rewritten as
		\begin{eqnarray}
		&& (1+2\lambda_{n})u_{m-1}^{n+1} \label{aa}\\
		&=&\lambda_{n}(u_{m-2}^{n+1}+u_{m}^{n+1})+u_{m-1}^{n}+\tau_{n}(u_{m-1}^{n})^{p}
	     -\dfrac{\tau_{n}}{(2h_{n})^{q}}\left|u_{m}^{n}-u_{m-2}^{n}\right|^{q-1}\left|u_{m}^{n+1}-u_{m-2}^{n+1}\right|.\nonumber \\ \nonumber 	
		\end{eqnarray}
		Using positivity and monotony we get
		\begin{equation*}
		(1+2\lambda_{n})u_{m-1}^{n+1} \geq  \lambda_{n}u_{m}^{n+1}+u_{m-1}^{n}-\dfrac{\tau_{n}}{(2h_{n})^{q}}(u_{m}^{n}-u_{m-2}^{n})^{q-1}(u_{m}^{n+1}-u_{m-2}^{n+1}).
		\end{equation*}
		We use that 
		\begin{equation}
		u_{m}^{n}-u_{m-2}^{n}\leq 2 u_{m}^{n} \text{ and } u_{m}^{n+1}-u_{m-2}^{n+1}\leq 2 u_{m}^{n+1},
		\label{bb}
		\end{equation}
		we obtain
		\begin{equation}
		u_{m-1}^{n+1}\geq \dfrac{\lambda_{n}}{1+2\lambda_{n}}u_{m}^{n+1}+\dfrac{1}{1+2\lambda_{n}}u_{m-1}^{n}-\dfrac{\tau_{n}}{h_{n}^{q}(1+2\lambda_{n})}(u_{m}^{n})^{q-1}u_{m}^{n+1}.
		\label{eq42}
		\end{equation}
		Furthermore from \eqref{aa} for $j=m$, we have
		\begin{equation}
		u_{m}^{n+1}=\dfrac{2\lambda_{n}}{1+2\lambda_{n}}u_{m-1}^{n+1}+\dfrac{u_{m}^{n}}{1+2\lambda_{n}}\left(1+\tau_{n}(u_{m}^{n})^{p-1}\right), 
		\label{eq43}
		\end{equation} 
		which implies that
		\begin{equation}
		u_{m}^{n+1}\geq \dfrac{u_{m}^{n}}{1+2\lambda_{n}}. 
		\label{eq44}
		\end{equation} 
		Using \eqref{eq42}, \eqref{eq43} and \eqref{eq44} we get for $p=2$ and $q=1$
			\begin{equation*}
			u_{m-1}^{n+1}\geq  \frac{\lambda_{n}}{(1+2\lambda_{n})^{2}}u_{m}^{n}+\frac{1}{1+2\lambda_{n}}u_{m-1}^{n}
			-\frac{\tau_{n}}{h_{n}(1+2\lambda_{n})}\left[\frac{2\lambda_{n}}{1+2\lambda_{n}}u_{m-1}^{n+1}+\frac{u_{m}^{n}}{1+2\lambda_{n}}(1+\tau_{n}u_{m}^{n})\right].
			\end{equation*}
		Then, 
			\begin{equation*}
			\left(1+\frac{2\tau_{n}\lambda_{n}}{h_{n}(1+2\lambda_{n})^{2}}\right)u_{m-1}^{n+1}\geq \frac{\lambda_{n}}{(1+2\lambda_{n})^{2}}u_{m}^{n}+\frac{1}{1+2\lambda_{n}}u_{m-1}^{n}
			-\frac{\tau_{n}u_{m}^{n}(1+\tau_{n}u_{m}^{n})}{h_{n}(1+2\lambda_{n})^{2}},
			\end{equation*}
		which implies that
		\begin{equation}
		u_{m-1}^{n+1}\geq \frac{\lambda_{n}h_{n}u_{m}^{n}+h_{n}(1+2\lambda_{n})u_{m-1}^{n}-\tau_{n}u_{m}^{n}(1+\tau_{n}u_{m}^{n})}{h_{n}(1+2\lambda_{n})^{2}+2\tau_{n}\lambda_{n}}.
		\label{eq45}
		\end{equation}
		Since the solution blows up, then we have $u_{m}^{n}>1$, moreover
		\begin{equation*}
		\tau_{n}=\frac{\tau}{u_{m}^{n}}  \text{ and } \  h_{n}=\min(\sqrt{2},h)=h.
		\end{equation*}
		Then
		\begin{equation*}
		\lambda_{n}=\frac{\tau}{h^{2}u_{m}^{n}}, \ \ \ h_{n}\lambda_{n}=\frac{\tau}{hu_{m}^{n}} \text{ and } \  \tau_{n}\lambda_{n}=\frac{\tau^{2}}{h^{2}(u_{m}^{n})^{2}}.
		\end{equation*}
		Hence, \eqref{eq45} implies
		\begin{equation*}
		u_{m-1}^{n+1}\geq \dfrac{\dfrac{\tau}{h}+h\left(1+\dfrac{2\tau}{h^{2}u_{m}^{n}}\right)u_{m-1}^{n}-\tau(1+\tau)}{h\left(1+\dfrac{2\tau}{h^{2}u_{m}^{n}}\right)^{2}+\dfrac{2\tau^{2}}{(u_{m}^{n})^{2}h^{2}}}.
		\end{equation*}
		As we have
		\begin{equation*}
		\lim_{n\rightarrow +\infty}u_{m}^{n}=+\infty,
		\end{equation*}
		then we get
		\begin{equation*}
		\lim_{n\rightarrow +\infty}u_{m-1}^{n+1}\geq \dfrac{\dfrac{\tau}{h}+h \lim\limits_{n\rightarrow +\infty}u_{m-1}^{n}-\tau(1+\tau)}{h}.
		\end{equation*}
		If we assume that $\lim\limits_{n\rightarrow +\infty}u_{m-1}^{n}\neq +\infty$, let $l=\lim\limits_{n\rightarrow +\infty}u_{m-1}^{n}$, then we have
		\begin{equation*}
		l\geq \dfrac{\tau}{h^{2}}+l-\dfrac{\tau(1+\tau)}{h}
		\Rightarrow \dfrac{\tau}{h}(\dfrac{1}{h}-(1+\tau))\leq 0,
		\end{equation*}
		which is a contradiction because $h<\dfrac{1}{1+\tau}.$\\
		Therefore, we have $$\lim_{n\rightarrow +\infty} u_{m-1}^{n}=+\infty,$$ and using symmetry we get the result of Theorem 2.1.
	\end{proof}
	The next important result for this paper is mentioned in the next theorem:
	\begin{th1}
		Let $U^{n}$ be the solution of \eqref{approchee}, we suppose that $p\geq 2$ and $1\leq q<\dfrac{2(p-1)}{p}.$\\
		\textbf{(a)} If $p=2$ and $q=1$ then 
		\begin{equation*}
		\lim_{n\rightarrow +\infty} u_{m-2}^{n}< +\infty.
		\end{equation*}
		\textbf{(b)} If $p>2$ and $q<\dfrac{2(p-1)}{p}$ then 
		\begin{equation*}
		\lim_{n\rightarrow +\infty} u_{m-1}^{n}< +\infty.
		\end{equation*}
	\end{th1}
	\begin{proof}
		Let prove \textbf{(a):} In \eqref{approchee}, if we take $p=2,\ q=1$ and $j=m-2$, we get
		\begin{eqnarray*}
			\frac{u_{m-2}^{n+1}-u_{m-2}^{n}}{\tau_{n}}&=&\frac{u_{m-1}^{n+1}-2u_{m-2}^{n+1}+u_{m-3}^{n+1}}{h_{n}^{2}}+\left(u_{m-2}^{n}\right)^{2}-\frac{1}{2h_{n}}\left(u_{m-1}^{n+1}-u_{m-3}^{n+1}\right)\\
			&\leq& \frac{u_{m-1}^{n+1}-2u_{m-2}^{n+1}+u_{m-3}^{n+1}}{h_{n}^{2}}+\left(u_{m-2}^{n}\right)^{2},
		\end{eqnarray*}
		but $u_{m-3}^{n+1}-u_{m-2}^{n+1}<0$, then
		\begin{equation*}
		\dfrac{u_{m-2}^{n+1}-u_{m-2}^{n}}{\tau_{n}}\leq \dfrac{u_{m-1}^{n+1}-u_{m-2}^{n+1}}{h_{n}^{2}}+\left(u_{m-2}^{n}\right)^{2},
		\end{equation*}
		which implies that 
		\begin{equation}
		(1+\lambda_{n})u_{m-2}^{n+1}\leq \lambda_{n}u_{m-1}^{n+1}+\left(1+\tau_{n}u_{m-2}^{n}\right)u_{m-2}^{n}.
		\label{eq46}
		\end{equation}
		In the other hand, in \eqref{aa} if we take $j=m-1$, we get
		\begin{equation*}
		\frac{u_{m-1}^{n+1}-u_{m-1}^{n}}{\tau_{n}}\leq \frac{u_{m-2}^{n+1}-2u_{m-1}^{n+1}+u_{m}^{n+1}}{h_{n}^{2}}+\left(u_{m-1}^{n}\right)^{2},
		\end{equation*}
		but $u_{m-2}^{n+1}-u_{m-1}^{n+1}<0$, then
		\begin{equation*}
		\frac{u_{m-1}^{n+1}-u_{m-1}^{n}}{\tau_{n}}\leq \frac{-u_{m-1}^{n+1}+u_{m}^{n+1}}{h_{n}^{2}}+\left(u_{m-1}^{n}\right)^{2},
		\end{equation*}
		which implies that 
		\begin{equation*}
		(1+\lambda_{n})u_{m-1}^{n+1}\leq \lambda_{n}u_{m}^{n+1}+\left(1+\tau_{n}u_{m-1}^{n}\right)u_{m-1}^{n},
		\end{equation*}
		and then 
		\begin{equation}
		u_{m-1}^{n+1}\leq \frac{\lambda_{n}u_{m}^{n+1}+\left(1+\tau_{n}u_{m-1}^{n}\right)u_{m-1}^{n}}{1+\lambda_{n}}.
		\label{eq47}
		\end{equation}
		Next, if we recall \eqref{eq43} for $p=2$ we get
		\begin{equation}
		u_{m}^{n+1}=\frac{2\lambda_{n}}{1+2\lambda_{n}}u_{m-1}^{n+1}+\frac{1+\tau_{n}u_{m}^{n}}{1+2\lambda_{n}}u_{m}^{n}.
		\label{eq48}
		\end{equation}
		Putting \eqref{eq48} in \eqref{eq47} we get
		\begin{equation*}
		u_{m-1}^{n+1}\leq \frac{\lambda_{n}}{1+\lambda_{n}}\left[\frac{2\lambda_{n}}{1+2\lambda_{n}}u_{m-1}^{n+1}+\frac{1+\tau_{n}u_{m}^{n}}{1+2\lambda_{n}}u_{m}^{n}\right]+\frac{\left(1+\tau_{n}u_{m-1}^{n}\right)}{1+\lambda_{n}}u_{m-1}^{n},
		\end{equation*}
		which implies that
		\begin{equation*}
		\left(1-\frac{2\lambda_{n}^{2}}{(1+\lambda_{n})(1+2\lambda_{n})}\right)u_{m-1}^{n+1}\leq \frac{\lambda_{n}(1+\tau_{n}u_{m}^{n})}{(1+2\lambda_{n})(1+\lambda_{n})}u_{m}^{n}+\frac{1+\tau_{n}u_{m-1}^{n}}{1+\lambda_{n}}u_{m-1}^{n},
		\end{equation*}
		and then
		\begin{equation} 
		u_{m-1}^{n+1}\leq \frac{\lambda_{n}(1+\tau_{n}u_{m}^{n})u_{m}^{n}+(1+2\lambda_{n})(1+\tau_{n}u_{m-1}^{n})u_{m-1}^{n}}{1+3\lambda_{n}}.
		\label{eq49}
		\end{equation}
		Now, putting \eqref{eq49} in \eqref{eq46}, we get
		\begin{eqnarray*} 
			u_{m-2}^{n+1}&\leq& \frac{\lambda_{n}}{1+\lambda_{n}}\left[\frac{\lambda_{n}(1+\tau_{n}u_{m}^{n})u_{m}^{n}+(1+2\lambda_{n})(1+\tau_{n}u_{m-1}^{n})u_{m-1}^{n}}{1+3\lambda_{n}}\right]+\frac{\left(1+\tau_{n}u_{m-2}^{n}\right)u_{m-2}^{n}}{1+\lambda_{n}}\\ &=&\frac{\left(1+\tau_{n}u_{m-2}^{n}\right)u_{m-2}^{n}}{1+\lambda_{n}}+\frac{\lambda_{n}^{2}(1+\tau_{n}u_{m}^{n})u_{m}^{n}+\lambda_{n}(1+2\lambda_{n})(1+\tau_{n}u_{m-1}^{n})u_{m-1}^{n}}{(1+\lambda_{n})(1+3\lambda_{n})}.
		\end{eqnarray*}
		Then
		\begin{equation}
		u_{m-2}^{n+1}\leq A_{n}u_{m-2}^{n}+B_{n},
		\label{au}
		\end{equation}
		here we have put
		\begin{equation*}
		A_{n}=\frac{\left(1+\tau_{n}u_{m-2}^{n}\right)}{1+\lambda_{n}}
		\end{equation*}
		and 
		\begin{equation*}
		B_{n}=\frac{\lambda_{n}^{2}(1+\tau_{n}u_{m}^{n})u_{m}^{n}+\lambda_{n}(1+2\lambda_{n})(1+\tau_{n}u_{m-1}^{n})u_{m-1}^{n}}{(1+\lambda_{n})(1+3\lambda_{n})}.
		\end{equation*}
		Then the inequality \eqref{au}
		implies by iterations that
		\begin{eqnarray*}
			u_{m-2}^{n}&\leq & A_{n-1}u_{m-2}^{n-1}+B_{n-1}\\
			&\leq& A_{n-1}A_{n-2}u_{m-2}^{n-2}+A_{n-1}B_{n-2}+B_{n-1}\\
			&&\vdots\\
			&\leq& u_{m-2}^{0}\prod_{k=0}^{n-1}{A_{k}}+\sum_{k=0}^{n-2}\left({B_{k}}\prod_{i=k+1}^{n-1}{A_{i}}\right)+B_{n-1} \\
			&\leq& u_{m-2}^{0}\prod_{k=0}^{n}{A_{k}}+\sum_{k=0}^{n-2}{ B_{k}}\prod_{i=0}^{n-1}{A_{i}}+B_{n-1}\\
			&\leq& u_{m-2}^{0}\prod_{k=0}^{n}{A_{k}}+\sum_{k=0}^{n-2}{ B_{k}}\prod_{k=0}^{n}{A_{k}}+B_{n-1}\left(\prod_{k=0}^{n}{A_{k}}\right)\\
			&\leq& u_{m-2}^{0}\prod_{k=0}^{n}{A_{k}}+\sum_{k=0}^{n-1}{ B_{k}}\prod_{k=0}^{n}{A_{k}}\\
			&\leq& u_{m-2}^{0}\prod_{k=0}^{n}{A_{k}}+\sum_{k=0}^{n}{ B_{k}}\prod_{k=0}^{n}{A_{k}}\\
			&\leq& \left(u_{m-2}^{0}+\sum_{k=0}^{n}{ B_{k}}\right)\prod_{k=0}^{n}{A_{k}}.
		\end{eqnarray*}
		To ensure boundedness of $u_{m-2}^{n}$ we shall prove that 
		\begin{equation*}
		\sum_{n\geq 0}{ B_{n}}<+\infty \text{ and } \prod_{n\geq0}{A_{n}}<+\infty.
		\end{equation*}
		To do this, we need the next lemma:
		\begin{lem}
			We define the sequence $a_{n}=\dfrac{u_{m-1}^{n}}{u_{m}^{n}}.$
			\begin{enumerate}
				\item For $p=2$ and $q=1$, we assume that $\sup\limits_{n}u_{m-1}^{n}>\dfrac{3}{h^{2}}(1+\tau),$ then $(a_{n})_{n}$ converges to 0.
				\item For $p>2$ and $q<\dfrac{2(p-1)}{p}$, we have
				\begin{enumerate}
					\item $(a_{n})_{n}$ converges to 0.
					\item $\lim\limits_{n\rightarrow +\infty}\dfrac{a_{n+1}}{a_{n}}=\dfrac{1}{1+\tau}.$
					\item $\lim\limits_{n\rightarrow +\infty}\dfrac{u^{n+1}_{m}}{u^{n}_{m}}=1+\tau>1.$
				\end{enumerate}
			\end{enumerate}
		\end{lem}
		\begin{proof}
			First of all, we look for some useful relations between $a_{n}$ and $a_{n+1}$. We recall \eqref{aa} for $p\geq 2$ and $1\leq q< \dfrac{2(p-1)}{p}<\dfrac{2p}{p+1}.$ We use the same calculations as \eqref{eq49} we obtain that \eqref{aa} implies
			\begin{equation}
			u_{m-1}^{n+1}\leq \frac{\lambda_{n}(1+\tau_{n}(u_{m}^{n})^{p-1})u_{m}^{n}+(1+2\lambda_{n})(1+\tau_{n}(u_{m-1}^{n})^{p-1})u_{m-1}^{n}}{1+3\lambda_{n}}.
			\label{eq411}
			\end{equation}
			Using \eqref{eq43}, we get
			\begin{eqnarray}
			\nonumber a_{n+1}& =&\dfrac{u_{m-1}^{n+1}}{u_{m}^{n+1}} \\
			\nonumber &=& \dfrac{1+2\lambda_{n}}{\dfrac{2\lambda_{n}u_{m-1}^{n+1}+(1+\tau_{n}(u_{m}^{n})^{p-1})u_{m}^{n}}{u_{m-1}^{n+1}}}\\
			&= &\dfrac{1+2\lambda_{n}}{2\lambda_{n}+\dfrac{(1+\tau_{n}(u_{m}^{n})^{p-1})u_{m}^{n}}{u_{m-1}^{n+1}}}.
			\label{star}
			\end{eqnarray}
			By substituting \eqref{eq411} into \eqref{star} we get:
			\begin{eqnarray*}
				a_{n+1}& \leq& (1+2\lambda_{n}) \left\{2\lambda_{n}+\frac{(1+3\lambda_{n})(1+\tau_{n}(u_{m}^{n})^{p-1})u_{m}^{n}}{\lambda_{n}(1+\tau_{n}(u_{m}^{n})^{p-1})u_{m}^{n}+(1+2\lambda_{n})(1+\tau_{n}(u_{m-1}^{n})^{p-1})u_{m-1}^{n}}\right\}^{-1}  \\ 
				& = & \frac{\lambda_{n}(1+2\lambda_{n})(1+\tau_{n}(u_{m}^{n})^{p-1})u_{m}^{n}+(1+2\lambda_{n})^{2}(1+\tau_{n}(u_{m-1}^{n})^{p-1})u_{m-1}^{n}}{(1+3\lambda_{n}+2\lambda_{n}^{2})(1+\tau_{n}(u_{m}^{n})^{p-1})u_{m}^{n}+2\lambda_{n}(1+2\lambda_{n})(1+\tau_{n}(u_{m-1}^{n})^{p-1})u_{m-1}^{n}}\\
				& \leq&
				\frac{\lambda_{n}(1+\tau_{n}(u_{m}^{n})^{p-1})u_{m}^{n}+(1+2\lambda_{n})(1+\tau_{n}(u_{m-1}^{n})^{p-1})u_{m-1}^{n}}{(1+\lambda_{n})(1+\tau_{n}(u_{m}^{n})^{p-1})u_{m}^{n}+2\lambda_{n}(1+\tau_{n}(u_{m-1}^{n})^{p-1})u_{m-1}^{n}}\\
				& \leq&
				\frac{\lambda_{n}(1+\tau_{n}(u_{m}^{n})^{p-1})+(1+2\lambda_{n})(1+\tau_{n}(u_{m-1}^{n})^{p-1})a_{n}}{(1+\lambda_{n})(1+\tau_{n}(u_{m}^{n})^{p-1})+2\lambda_{n}(1+\tau_{n}(u_{m-1}^{n})^{p-1})a_{n}}.
			\end{eqnarray*}	
			But we have $\tau_{n}=\dfrac{\tau}{(u_{m}^{n})^{p-1}}$, then
			\begin{equation}
			a_{n+1}\leq \frac{\lambda_{n}(1+\tau)+(1+2\lambda_{n})(1+\tau(a_{n})^{p-1})a_{n}}{(1+\lambda_{n})(1+\tau)+2\lambda_{n}(1+\tau(a_{n})^{p-1})a_{n}}.
			\label{eq412}
			\end{equation}
			And finally we get 
			\begin{equation}
			\frac{a_{n+1}}{a_{n}}\leq
			\frac{\lambda_{n}(1+\tau)(a_{n})^{-1}+(1+2\lambda_{n})(1+\tau(a_{n})^{p-1})}{(1+\lambda_{n})(1+\tau)+2\lambda_{n}(1+\tau(a_{n})^{p-1})a_{n}}. 
			\label{eq413}
			\end{equation}
			In the other hand, using \eqref{aa} and \eqref{bb} we get
			\begin{equation*}
			u_{m-1}^{n+1} \geq
			\frac{\lambda_{n}u_{m}^{n+1}+(1+\tau_{n}(u_{m-1}^{n})^{p-1})u_{m-1}^{n}}{1+2\lambda_{n}}-\frac{\tau_{n}}{h_{n}^{q}(1+2\lambda_{n})}(u_{m}^{n})^{q-1}u_{m}^{n+1}.
			\end{equation*}
			By using \eqref{eq43}, we have
			\begin{align*}
			u_{m-1}^{n+1}
			& \geq
			\frac{2\lambda_{n}^{2}u_{m-1}^{n+1}+\lambda_{n}(1+\tau_{n}(u_{m}^{n})^{p-1})u_{m}^{n}+(1+2\lambda_{n})(1+\tau_{n}(u_{m-1}^{n})^{p-1})u_{m-1}^{n}}{(1+2\lambda_{n})^{2}}\\
			&\ \ \ \ \ \  -\frac{2\lambda_{n}\tau_{n}(u_{m}^{n})^{q-1}u_{m-1}^{n+1}}{h_{n}^{q}(1+2\lambda_{n})^{2}}-\frac{\tau_{n}(1+\tau_{n}(u_{m}^{n})^{p-1})(u_{m}^{n})^{q}}{h_{n}^{q}(1+2\lambda_{n})^{2}},
			\end{align*}
			which implies
			\begin{eqnarray*}
				&&\left(1-\frac{2\lambda_{n}^{2}}{(1+2\lambda_{n})^{2}}+\frac{2\lambda_{n}\tau_{n}(u_{m}^{n})^{q-1}}{h_{n}^{q}(1+2\lambda_{n})^{2}}\right)u_{m-1}^{n+1}\\
				&\geq& \frac{\lambda_{n}(1+\tau_{n}(u_{m}^{n})^{p-1})u_{m}^{n}+(1+2\lambda_{n})(1+\tau_{n}(u_{m-1}^{n})^{p-1})u_{m-1}^{n}}{(1+2\lambda_{n})^{2}}\\ &&\ \ \ \ -\frac{\tau_{n}(1+\tau_{n}(u_{m}^{n})^{p-1})(u_{m}^{n})^{q}}{h_{n}^{q}(1+2\lambda_{n})^{2}},
			\end{eqnarray*}
			and then
			\begin{eqnarray}
			u_{m-1}^{n+1}
			\nonumber &\geq&
			\frac{\lambda_{n}h_{n}^{q}(1+\tau_{n}(u_{m}^{n})^{p-1})u_{m}^{n}+h_{n}^{q}(1+2\lambda_{n})(1+\tau_{n}(u_{m-1}^{n})^{p-1})u_{m-1}^{n}}{h_{n}^{q}(1+2\lambda_{n})^{2}-2h_{n}^{q}\lambda_{n}^{2}+2\lambda_{n}\tau_{n}(u_{m}^{n})^{q-1}}\\
			&& \ \ -\frac{\tau_{n}(1+\tau_{n}(u_{m}^{n})^{p-1})(u_{m}^{n})^{q}}{{h_{n}^{q}(1+2\lambda_{n})^{2}-2h_{n}^{q}\lambda_{n}^{2}+2\lambda_{n}\tau_{n}(u_{m}^{n})^{q-1}}}.
			\label{etoile}
			\end{eqnarray}
			Using \eqref{eq43} and \eqref{etoile}, we get
				\begin{eqnarray*}
					&& a_{n+1}\\
					&=&\dfrac{u_{m-1}^{n+1}}{u_{m}^{n+1}}\\
					&= & (1+2\lambda_{n})\left\{2\lambda_{n}+\frac{(1+\tau_{n}(u_{m}^{n})^{p-1})u_{m}^{n}}{u_{m-1}^{n+1}}\right\}^{-1}\\
					&\geq & \frac{(1+\tau_{n}(u_{m}^{n})^{p-1})u_{m}^{n}\left(\lambda_{n}h_{n}^{q}-\tau_{n}(u_{m}^{n})^{q-1}\right)+h_{n}^{q}(1+2\lambda_{n})(1+\tau_{n}(u_{m-1}^{n})^{p-1})u_{m-1}^{n}}{2\lambda_{n}h_{n}^{q}(1+\tau_{n}(u_{m-1}^{n})^{p-1})u_{m-1}^{n}+h_{n}^{q}(1+2\lambda_{n})(1+\tau_{n}(u_{m}^{n})^{p-1})u_{m}^{n}}.
				\end{eqnarray*}
			Then we can deduce that
			\begin{equation*}
			\frac{a_{n+1}}{a_{n}}
			\geq
			\frac{(1+\tau_{n}(u_{m}^{n})^{p-1})(\lambda_{n}-\tau_{n}h_{n}^{-q}(u_{m}^{n})^{q-1})a_{n}^{-1}+(1+2\lambda_{n})(1+\tau_{n}(u_{m-1}^{n})^{p-1})}{(1+\tau_{n}(u_{m}^{n})^{p-1})(1+2\lambda_{n})+2\lambda_{n}(1+\tau_{n}(u_{m-1}^{n})^{p-1})a_{n}}.
			\end{equation*}
			Finally we get
			\begin{equation}
			\frac{a_{n+1}}{a_{n}}
			\geq
			\frac{(1+\tau)(\lambda_{n}-\tau h_{n}^{-q}(u_{m}^{n})^{q-p})a_{n}^{-1}+(1+2\lambda_{n})(1+\tau(a_{n})^{p-1})}{(1+\tau)(1+2\lambda_{n})+2\lambda_{n}(1+\tau(a_{n})^{p-1})a_{n}}.
			\label{eq414}
			\end{equation}
			Next, we prove that the sequence $(a_{n})_{n}$ converges to 0. To prove convergence, we only need to show that $\dfrac{a_{n+1}}{a_{n}}<1$.\\
			But 
			\begin{equation*}
			\frac{a_{n+1}}{a_{n}}\leq
			\frac{\lambda_{n}(1+\tau)(a_{n})^{-1}+(1+2\lambda_{n})(1+\tau(a_{n})^{p-1})}{(1+\lambda_{n})(1+\tau)+2\lambda_{n}(1+\tau(a_{n})^{p-1})a_{n}}. 
			\end{equation*}
			Let
			\begin{equation*}
			A:=(1+\tau)\lambda_{n}+(1+2\lambda_{n})(1+\tau(a_{n})^{p-1})a_{n}-(1+\lambda_{n})(1+\tau)a_{n}-2\lambda_{n}(1+\tau(a_{n})^{p-1})a_{n}^{2}.\\ 
			\end{equation*}
			We shall prove that $A< 0.$\\
			\textbf{(1)} First of all, we can see that, for $p=2$ and $q=1$
			\begin{eqnarray*}
			A&=&(1+\tau)\lambda_{n}+(1+2\lambda_{n})(1+\tau a_{n})a_{n}-(1+\lambda_{n})(1+\tau)a_{n}-2\lambda_{n}(1+\tau a_{n})a_{n}^{2}\\
			&=& (1+\tau)\lambda_{n}+(1+\tau a_{n})a_{n}+2\lambda_{n}(1+\tau a_{n})a_{n}-(1+\tau)a_{n}-\lambda_{n}(1+\tau)a_{n}-2\lambda_{n}(1+\tau a_{n})a_{n}^{2}\\
			&=&\lambda_{n}\left(1+\tau+2a_{n}(1+\tau a_{n})-(1+\tau)a_{n}-2a_{n}^{2}(1+\tau a_{n}) \right)+\tau a_{n}^{2}-\tau a_{n}\\ 
			&=&\lambda_{n}\left((1+\tau)(1-a_{n})+2a_{n}(1+\tau a_{n})(1-a_{n}) \right)+\tau a_{n}(a_{n}-1)\\
		    &=&(1-a_{n})\lambda_{n}(1+\tau +2a_{n}(1+\tau a_{n}))+\tau a_{n}(a_{n}-1)\\
			\end{eqnarray*}
			Using
			\begin{equation*}
			\left\{
			\begin{array}{lll}
			a_{n}=\dfrac{u_{m-1}^{n}}{u_{m}^{n}},\\
			0<a_{n}<1,\\
			a_{n}<u_{m-1}^{n},\\
			\tau=\tau_{n}u_{m}^{n},\\
			h_{n}=h,\\
			\end{array}
			\right.
			\end{equation*}
			we get
			\begin{eqnarray*}
				A&=&(1-a_{n})\lambda_{n}(1+\tau +2a_{n}(1+\tau a_{n})+\tau a_{n}(a_{n}-1) \\
				& <& \lambda_{n}(1-a_{n})(1+\tau +2(1+\tau))+\tau_{n} u_{m-1}^{n}(a_{n}-1)\\
				& < &(1-a_{n})\tau_{n} (3h^{-2}(1+\tau)-u_{m-1}^{n})\\
			\end{eqnarray*}
			Using the condition: $\sup\limits_{n}\left\{u_{m-1}^{n}\right\}>3h^{-2}(1+\tau)$,we can see that $A< 0$, so that $0\leq a_{n+1}< a_{n}<1$, which implies that $\lim\limits_{n\rightarrow +\infty}a_{n}=a$ exists and satisfies $0\leq a<1.$\\
			\textbf{(2)} For $p>2$ and $q<\dfrac{2p-2}{p},$ we can see that
			\begin{equation*}
			\lambda_{n}-(1+\lambda_{n})a_{n}<0,
			\end{equation*}
			if not, then, 
			\begin{equation*}
			\dfrac{\lambda_{n}}{a_{n}}\geq 1+\lambda_{n} \Rightarrow \frac{\tau 2^{\frac{-2}{2-q}}}{u_{m-1}^{n}}\left(u_{m}^{n}\right)^{\frac{2-2p+pq}{2-q}}\geq 1+\lambda_{n},
			\end{equation*}
			which is a contradiction because of $q<\frac{2p-2}{p}.$\\
			Let now, 
			\begin{equation*}
			A_{1}=\frac{1+\tau}{1+\tau a_{n}^{p-1}}>1
			\end{equation*}
			and
			\begin{equation*}
			A_{2}=\frac{2\lambda_{n}a_{n}-(1+2\lambda_{n})}{\lambda_{n}-(1+\lambda_{n})a_{n}}<1.
			\end{equation*}
			Then it is clear that
			\begin{eqnarray*}
				&& A_{1}> a_{n}A_{2}.\\
				&&\Rightarrow \frac{1+\tau}{1+\tau a_{n}^{p-1}} > \frac{a_{n}(2\lambda_{n}a_{n}-(1+2\lambda_{n}))}{\lambda_{n}-(1+\lambda_{n})a_{n}}.\\
				&&\Rightarrow (1+\tau)(\lambda_{n}-(1+\lambda_{n})a_{n}) < a_{n} (1+\tau a_{n}^{p-1})(2\lambda_{n}a_{n}-(1+2\lambda_{n})).\\
				&&\Rightarrow (1+\tau)\lambda_{n}+(1+2\lambda_{n})a_{n}(1+\tau a_{n}^{p-1})-(1+\tau)(1+\lambda_{n}a_{n})-2\lambda_{n}a_{n}^{2}(1+\tau a_{n}^{p-1})< 0.\\
				&&\Rightarrow A< 0.
			\end{eqnarray*}
			So $0\leq a_{n+1} < a_{n} < 1.$\\
			We shall prove now that $a=0$ for all $p>1$ and $1\leq q<\dfrac{2(p-1)}{p}$. By reduction to absurdity we suppose that $0<a<1$. Letting $n\rightarrow \infty $ in \eqref{eq412} we obtain 
			\begin{equation*}
			a \leq \dfrac{1+\tau a^{p-1}}{1+\tau}a <a
			\end{equation*}
			which is a contradiction. This proves that $a=0.$\\
			Next we prove that $\lim\limits_{n\rightarrow+\infty}\dfrac{a_{n+1}}{a_{n}}=\dfrac{1}{1+\tau}$, for $p>2$ and $q<\dfrac{2(p-1)}{p}$\\
			By means of \eqref{eq413} we get
			\begin{equation}
			\frac{a_{n+1}}{a_{n}}\leq
			\frac{\lambda_{n}(1+\tau)(a_{n})^{-1}+(1+2\lambda_{n})(1+\tau(a_{n})^{p-1})}{(1+\lambda_{n})(1+\tau)+2\lambda_{n}(1+\tau(a_{n})^{p-1})a_{n}},
			\label{eq416}
			\end{equation}
			but
			\begin{equation*}
			\lambda_{n}(1+\tau)(a_{n})^{-1}=c_{1}(u_{m-1}^{n})^{-1}(u_{m}^{n})^{\frac{-2p+pq+2}{2-q}},
			\end{equation*} 
			where $c_{1}=\dfrac{\tau(1+\tau)}{4^{\frac{1}{2-q}}}.$\\
			And for $q<\dfrac{2p-2}{p}$ we have $\dfrac{-2p+pq+2}{2-q}<0$, then we obtain 
			\begin{equation}
			\lim_{n\rightarrow +\infty}\dfrac{a_{n+1}}{a_{n}}\leq \dfrac{1}{1+\tau}.
			\label{h}
			\end{equation} 
			In the other hand, using \eqref{eq414} we get
			\begin{equation*}
			\frac{a_{n+1}}{a_{n}}
			\geq
			\frac{(1+\tau)(\lambda_{n}-\tau h_{n}^{-q}(u_{m}^{n})^{q-p})a_{n}^{-1}+(1+2\lambda_{n})(1+\tau(a_{n})^{p-1})}{(1+\tau)(1+2\lambda_{n})+2\lambda_{n}(1+\tau(a_{n})^{p-1})a_{n}},
			\label{eq417}
			\end{equation*}
			but
			\begin{equation*}
			(\lambda_{n}-\tau h_{n}^{-q}(u_{m}^{n})^{q-p})a_{n}^{-1}=(u_{m-1}^{n})^{-1}\left(c_{1}\left(u_{m}^{n}\right)^{-p+2+\frac{2q-2}{2-q}}-c_{2}\left(u_{m}^{n}\right)^{q-p+1+\frac{-q(-q+1)}{2-q}}\right),
			\end{equation*} 
			where $c_{1},c_{2} \in \mathbb{R}.$\\
			And for $q<\dfrac{2(p-1)}{p}$ we have
			\begin{equation*}
			 \left(\lambda_{n}-\tau h_{n}^{-q}(u_{m}^{n})^{q-p}\right)a_{n}^{-1}\rightarrow 0 \text{ as } n\rightarrow +\infty,
			 \end{equation*}
			 then we obtain 
			\begin{equation}
			\lim_{n\rightarrow +\infty}\dfrac{a_{n+1}}{a_{n}}\geq \dfrac{1}{1+\tau}.
			\label{hh}
			\end{equation} 
			Finally from \eqref{h} and \eqref{hh} we deduce that
			\begin{equation*}
			\lim_{n\rightarrow +\infty}\dfrac{a_{n+1}}{a_{n}}= \dfrac{1}{1+\tau}<1.
			\end{equation*} 
			To finish the proof of Lemma 2.3 we shall prove that for all $p>2$ and $q<\dfrac{2(p-1)}{p}$ we have $\lim\limits_{n\rightarrow +\infty}\dfrac{u_{m}^{n+1}}{u_{m}^{n}}=1+\tau.$\\
			From \eqref{eq43}, we know that
			\begin{equation*}
			(1+2\lambda_{n})u_{m}^{n+1}-2\lambda_{n}u_{m-1}^{n+1}=(1+\tau_{n}(u_{m}^{n})^{p-1})u_{m}^{n},
			\end{equation*}
			which implies
			\begin{equation*}
			1+2\lambda_{n}-2\lambda_{n}\frac{u_{m-1}^{n+1}}{u_{m}^{n+1}}=(1+\tau_{n}(u_{m}^{n})^{p-1})\frac{u_{m}^{n}}{u_{m}^{n+1}}. 
			\end{equation*}
			Then
			\begin{equation*}
			1+2\lambda_{n}-2\lambda_{n}a_{n+1}=(1+\tau)\frac{u_{m}^{n}}{u_{m}^{n+1}}. 
			\end{equation*}
			So
			\begin{equation*}
			1=(1+\tau)\lim_{n\rightarrow +\infty}\frac{u_{m}^{n}}{u_{m}^{n+1}}.
			\end{equation*}
			This implies  $\lim\limits_{n\rightarrow +\infty}\dfrac{u_{m}^{n+1}}{u_{m}^{n}}=1+\tau>1.$
			This achieve the proof of Lemma 2.3.
		\end{proof}
		Now we can finish the proof of Theorem 2.2. We have showed that
		\begin{equation*} 
		u_{m-2}^{n} \leq \left(u_{m-2}^{0}+\sum_{k=0}^{n}{ B_{k}}\right)\prod_{k=0}^{n}{A_{k}},
		\end{equation*}
		where 
		\begin{equation*}
		A_{n}=\frac{\left(1+\tau_{n}u_{m-2}^{n}\right)}{1+\lambda_{n}}
		\end{equation*}
		and 
		\begin{equation*}
		B_{n}=\frac{\lambda_{n}^{2}(1+\tau_{n}u_{m}^{n})u_{m}^{n}+\lambda_{n}(1+2\lambda_{n})(1+\tau_{n}u_{m-1}^{n})u_{m-1}^{n}}{(1+\lambda_{n})(1+3\lambda_{n})}.
		\end{equation*} 
		Using that $u_{m}^{n}>>1$, we can see that for $p=2$ and $q=1$ we have 
		\begin{equation*}
		\tau_{n}=\dfrac{\tau}{u_{m}^{n}} \text{ and } h_{n}=h.
		\end{equation*}
		Then
		\begin{equation*}
		A_{n} \leq 1+\tau_{n}u_{m-2}^{n}= 1+\tau \frac{u_{m-2}^{n}}{u_{m}^{n}}\leq 1+\tau \frac{u_{m-1}^{n}}{u_{m}^{n}}=1+\tau a_{n},
		\end{equation*} 
		and
		\begin{align*}
		B_{n}& =\frac{\lambda_{n}^{2}(1+\tau_{n}u_{m}^{n})u_{m}^{n}+\lambda_{n}(1+2\lambda_{n})(1+\tau_{n}u_{m-1}^{n})u_{m-1}^{n}}{(1+\lambda_{n})(1+3\lambda_{n})} \\
		& \leq \lambda_{n}^{2}(1+\tau_{n}u_{m}^{n})u_{m}^{n}+\lambda_{n}(1+2\lambda_{n})(1+\tau_{n}u_{m-1}^{n})u_{m-1}^{n}\\
		& \leq \lambda_{n}^{2}(1+\tau)u_{m}^{n}+\lambda_{n}(1+2\lambda_{n})(1+\tau_{n}u_{m}^{n})u_{m-1}^{n}\\
		& = \frac{\tau^{2}}{h^{4}}(1+\tau)\frac{1}{u_{m}^{n}}+\frac{\tau}{h^{2}u_{m}^{n}}(1+2\frac{\tau}{h^{2}u_{m}^{n}})(1+\tau)u_{m-1}^{n}\\
		& \leq c^{2}(1+\tau)\frac{u_{m-1}^{n}}{u_{m}^{n}}+c(1+2c)(1+\tau)\frac{u_{m-1}^{n}}{u_{m}^{n}}\\
		&\leq c(1+\tau)\left(c+(1+2c)\right)\frac{u_{m-1}^{n}}{u_{m}^{n}}\\
		&\leq c(1+\tau)(1+3c)\frac{u_{m-1}^{n}}{u_{m}^{n}},
		\end{align*} 
		with $c=\dfrac{\tau}{h^{2}}.$\\
		But we have
		\begin{equation*}
		\lim_{n\rightarrow +\infty}\frac{a_{n+1}}{a_{n}}<1 \text{ and } a_{n}>0,
		\end{equation*}
		then
		\begin{equation*}
		0< \sum_{n\geq0}{a_{n}}< +\infty.
		\end{equation*}
		In the other hand, for all $c>0$, we have $\sum\limits_{n\geq0}{ca_{n}}< +\infty$, then
		\begin{equation*}
		1< \prod_{n\geq 0} (1+ca_{n})< +\infty. 
		\end{equation*}
		We deduce from this that
		\begin{equation*}
		0<\sum_{n\geq 0}{ B_{n}}\leq c(1+\tau)(1+3c)\sum_{n\geq0}{a_{n}}<+\infty,
		\end{equation*}
		and
		\begin{equation*}
		1<\prod_{n\geq0}{A_{n}}\leq \prod_{n\geq 0} (1+\tau a_{n})<+\infty,
		\end{equation*}
		which implies that 
		\begin{equation*}
		\lim_{n\rightarrow +\infty}u_{m-2}^{n}<+\infty.
		\end{equation*}
		Now we will prove the second result of Theorem 2.2, that is:
		\begin{equation*}
		\text{ If } p>2 \text{ and } q<\frac{2(p-1)}{p} \text{ then } \lim_{n\rightarrow +\infty} u_{m-1}^{n}< +\infty.
		\end{equation*}
		In \eqref{approchee}, we put $j=m-1$ and we consider the quantity
		\begin{align*}
		u_{m-1}^{n+1}-u_{m-1}^{n} & \leq \lambda_{n}(u_{m-2}^{n+1}-2u_{m-1}^{n+1}+u_{m}^{n+1})+\tau_{n}(u_{m-1}^{n})^{p}\\
		& =G_{n}+H_{n},
		\end{align*}
		where
		\begin{eqnarray*}
			G_{n}&=& \lambda_{n}(u_{m-2}^{n+1}-2u_{m-1}^{n+1}+u_{m}^{n+1})\\
			&=&c(u_{m}^{n})^{-p+1+\frac{2(q-1)}{2-q}}(u_{m-2}^{n+1}-2u_{m-1}^{n+1}+u_{m}^{n+1})\\
			&=&c(u_{m}^{n})^{-p+1+\frac{2(q-1)}{2-q}}u_{m}^{n+1}(\frac{u_{m-2}^{n+1}}{u_{m}^{n+1}}-2\frac{u_{m-1}^{n+1}}{u_{m}^{n+1}}+1)\\
			&=& c(u_{m}^{n})^{-p+2+\frac{2(q-1)}{2-q}}\frac{u_{m}^{n+1}}{u_{m}^{n}}(\frac{u_{m-2}^{n+1}}{u_{m-1}^{n+1}}a_{n+1}-2a_{n+1}+1)\\
			&=&c(u_{m}^{n})^{-p+2+\frac{2q-2}{2-q}}\frac{u_{m}^{n+1}}{u_{m}^{n}}(1-a_{n+1}(2-\frac{u_{m-2}^{n+1}}{u_{m-1}^{n+1}}))\\
			& >&0,
		\end{eqnarray*}
		with $c:=\dfrac{\tau}{2^{\frac{2}{2-q}}}$
		and 
		\begin{equation*}
		H_{n}=\tau_{n}(u_{m-1}^{n})^{p}=\tau u_{m}^{n}(a_{n})^{p}>0.
		\end{equation*}
		Therefore, using Lemma 2.3, we get
		\begin{equation*}
		\lim_{n\rightarrow +\infty}\dfrac{G_{n+1}}{G_{n}} =\left(1+\tau\right)^{-p+2+\frac{2q-2}{q-2}}<1 \text{ for } q<\dfrac{2(p-1)}{p},
		\end{equation*}
		which implies
		\begin{equation*}
		\sum_{n\geq0}{G_{n}}<+\infty.
		\end{equation*} 
		Also 
		\begin{equation*}
		\lim_{n\rightarrow +\infty}\frac{H_{n+1}}{H_{n}} =(1+\tau)^{-p+1}<1 \text{ for } p>2,
		\end{equation*}
		which implies
		\begin{equation*}
		\sum_{n\geq0}{H_{n}}<+\infty.
		\end{equation*} 
		Hence we get the boundedness of $u_{m-1}^{n}$ from:
		\begin{eqnarray*}
			0<u_{m-1}^{n}&=& \sum_{k=1}^{n}{(u_{m-1}^{k}-u_{m-1}^{k-1})}+u_{m-1}^{0} \\
			& \leq & \sum_{k=1}^{n}{(G_{k-1}+H_{k-1})}+u_{m-1}^{0} \\
			& \leq & \sum_{k=0}^{+\infty}{(G_{k}+H_{k})}+u_{m-1}^{0} \\
			& < & +\infty.
		\end{eqnarray*}
		Thus we have completed the proof of Theorem 2.2.
	\end{proof}
	\section{Convergence }
	\noindent In this section we prove the convergence of the numerical solution given by \eqref{approchee}, to the nodal values of the solution $u$ of \eqref{exacte} on each fixed interval time $[0,T],T<T^{*}$ as far as the smoothness of $u$ is guaranteed.
	\begin{lem}
		Let $u$ be the classical solution of \eqref{exacte} and $U^{n}$ be the numerical solution of \eqref{approchee}. Let $T$ be an arbitrary number such that $0<T<T^{*}$. Then there exist positive constants $C_{0},\ C_{1}$, depending only on $T$ and $u_{0}$, such that\\
		\textbf{(A)} For $p>2$ and $q<\dfrac{2(p-1)}{p}$
		\begin{equation*}
		\max_{1\leq j \leq m-2}\left|u_{j}^{n}-u(x_{j},t^{n})\right|\leq C_{0}h^{3-q}
		\end{equation*} 
		holds so far as $t_{n}<T.$\\
		\textbf{(B)} For $p>1$ and $q=1$
		\begin{equation*}
		\max_{1\leq j \leq m-1}\left|u_{j}^{n}-u(x_{j},t^{n})\right|\leq C_{1}h^{2}
		\end{equation*} 
		holds so far as $t_{n}<T.$
	\end{lem}
	\noindent Before studying local convergence, we prove the consistency of the scheme.
	\subsection{Consistency}
	For all $1\leq j \leq N_{n},$ we define
	\begin{eqnarray*}
		\epsilon_{j}^{n}& =&\frac{u(x_{j},t^{n+1})-u(x_{j},t^{n})}{\tau_{n}}-\frac{u(x_{j+1},t^{n+1})-2u(x_{j},t^{n+1})+u(x_{j-1},t^{n+1})}{h_{n}^{2}} \\ &-&\left(u(x_{j},t^{n})\right)^{p}+\left|\frac{u(x_{j+1},t^{n})-u(x_{j-1},t^{n})}{2h_{n}}\right|^{q-1}\left|\frac{u(x_{j+1},t^{n+1})-u(x_{j-1},t^{n+1}}{2h_{n}}\right|.
	\end{eqnarray*}
	We use Taylor formula, we obtain
	\begin{eqnarray}
		\frac{\partial u}{\partial t}(x_{j},t^{n})&=&\frac{u(x_{j},t^{n+1})-u(x_{j},t^{n})}{\tau_{n}}-\frac{\tau_{n}}{2}\frac{\partial^{2} u}{\partial t^{2}}(x_{j},t^{n}+\tau_{n}\theta_{1}).
		\label{etoi}\\
	\frac{\partial u}{\partial x}(x_{j},t^{n})&=&\frac{u(x_{j+1},t^{n})-u(x_{j-1},t^{n})}{2h_{n}}-\frac{h_{n}^{2}}{3}\frac{\partial^{3} u}{\partial x^{3}}(x_{j}+h_{n}\theta_{2},t^{n})\nonumber \\
	&&\ \ \ -\frac{h_{n}^{2}}{3}\frac{\partial^{3} u}{\partial x^{3}}(x_{j}-h_{n}\theta_{3},t^{n}).
	\label{E}\\
	\frac{\partial^{2} u}{\partial x^{2}}(x_{j},t^{n})&=&\frac{u(x_{j+1},t^{n})-2u(x_{j},t^{n})+u(x_{j-1},t^{n})}{h_{n}^{2}}-\frac{h_{n}^{2}}{24}\frac{\partial^{4} u}{\partial x^{4}}(x_{j}+h_{n}\theta_{4},t^{n}) \nonumber \\
	&&\ \ \ -\frac{h_{n}^{2}}{24}\frac{\partial^{4} u}{\partial x^{4}}(x_{j}-h_{n}\theta_{5},t^{n}).\nonumber \\
	\nonumber \frac{\partial^{2} u}{\partial x^{2}}(x_{j},t^{n})&=&\frac{\partial^{2} u}{\partial x^{2}}(x_{j},t^{n+1})-\tau_{n}\frac{\partial^{3} u}{\partial t \partial x^{2}}(x_{j},t^{n}+\tau_{n}\theta_{6})\\
	\nonumber &=&\frac{u(x_{j+1},t^{n+1})-2u(x_{j},t^{n+1})+u(x_{j-1},t^{n+1})}{h_{n}^{2}}+\frac{h_{n}^{2}}{24}\frac{\partial^{4} u}{\partial x^{4}}(x_{j}+h_{n}\theta_{7},t^{n+1})\\
	&&+\frac{h_{n}^{2}}{24}\frac{\partial^{4} u}{\partial x^{4}}(x_{j}-h_{n}\theta_{8},t^{n+1})+\tau_{n}\frac{\partial^{3} u}{\partial t \partial x^{2}}(x_{j},t^{n}+\tau_{n}\theta_{6}).
	\label{et}
	\end{eqnarray}
	\noindent where $0<\theta_{i}<1$ for $i=1,...,8.$\\
	We define
	\begin{equation*}
	F=\left|\frac{\partial u}{\partial x}(x_{j},t^{n})\right|^{q}-\left|\frac{u(x_{j+1},t^{n})-u(x_{j-1},t^{n})}{2h_{n}}\right|^{q}.
	\end{equation*}
	We use the mean value theorem, the monotony and the symmetry of the exact solution proved in Theorem 2.3 and Theorem 2.4 in \cite{hani}, then there exists $A$ between $\dfrac{\partial u}{\partial x}(x_{j},t^{n})$ and\\ $\dfrac{u(x_{j+1},t^{n})-u(x_{j-1},t^{n})}{2h_{n}}$ such that
	\begin{eqnarray*}
		&&\left| \left|\frac{\partial u}{\partial x}(x_{j},t^{n})\right|^{q}-\left|\frac{u(x_{j+1},t^{n})-u(x_{j-1},t^{n})}{2h_{n}}\right|^{q}\right| \\
		&=&\left| q\left|A\right|^{q-1}\left(\frac{\partial u}{\partial x}(x_{j},t^{n})-\frac{u(x_{j+1},t^{n})-u(x_{j-1},t^{n})}{2h_{n}}\right)\right| \\
		&=& q\left|A\right|^{q-1} o(h_{n}^{2}),
	\end{eqnarray*}
	with 
	\begin{equation*}
	\left|\frac{\partial u}{\partial x}(x_{j},t^{n})-A\right|\leq \left|\frac{\partial u}{\partial x}(x_{j},t^{n})-\frac{u(x_{j+1},t^{n})-u(x_{j-1},t^{n})}{2h_{n}}\right| \leq o(h_{n}^{2}).
	\end{equation*}
	Since $\left|\dfrac{\partial u}{\partial x}\right|$ is bounded before blow up by \cite{chipotweissler}, then we can deduce that $A$ is bounded too. So we can write that
	\begin{eqnarray}
	\left|\frac{\partial u}{\partial x}(x_{j},t^{n})\right|^{q}&=&\left|\frac{u(x_{j+1},t^{n})-u(x_{j-1},t^{n})}{2h_{n}}\right|^{q}+o(h_{n}^{2})\label{starstar}\\
	\nonumber  &=&\left|\frac{u(x_{j+1},t^{n})-u(x_{j-1},t^{n})}{2h_{n}}\right|^{q-1}\left|\frac{u(x_{j+1},t^{n})-u(x_{j-1},t^{n})}{2h_{n}}\right|+o(h_{n}^{2})\\
	\nonumber &=&\left|\frac{u(x_{j+1},t^{n})-u(x_{j-1},t^{n})}{2h_{n}}\right|^{q-1}\left|\frac{\partial u}{\partial x}(x_{j},t^{n})+o(h_{n}^{2})\right|+o(h_{n}^{2})\\
	\nonumber &=&\left|\frac{u(x_{j+1},t^{n})-u(x_{j-1},t^{n})}{2h_{n}}\right|^{q-1}\left|\frac{\partial u}{\partial x}(x_{j},t^{n})\right|+o(h_{n}^{2})\\
	\nonumber &=&\left|\frac{u(x_{j+1},t^{n})-u(x_{j-1},t^{n})}{2h_{n}}\right|^{q-1}\left|\frac{\partial u}{\partial x}(x_{j},t^{n+1})\right|+o(\tau_{n})+o(h_{n}^{2}).
	\end{eqnarray}
	Then
	\begin{eqnarray}
	\left|\frac{\partial u}{\partial x}(x_{j},t^{n})\right|^{q}&=&\left|\frac{u(x_{j+1},t^{n})-u(x_{j-1},t^{n})}{2h_{n}}\right|^{q-1}\left|\frac{u(x_{j+1},t^{n+1})-u(x_{j-1},t^{n+1})}{2h_{n}}\right| \nonumber \\
	&&\ \ \ \ +o(\tau_{n})+o(h_{n}^{2}).
	\label{etoietoi}
	\end{eqnarray}
	We replace \eqref{etoi}, \eqref{et} and \eqref{etoietoi} in $\epsilon_{j}^{n}$ we obtain
	\begin{eqnarray*}
		\epsilon_{j}^{n}&=&\frac{\partial u}{\partial t}(x_{j},t^{n})+\frac{\tau_{n}}{2}\frac{\partial^{2} u}{\partial t^{2}}(x_{j},t^{n}+\tau_{n}\theta_{1})-
		\frac{\partial^{2} u}{\partial x^{2}}(x_{j},t^{n})-\tau_{n}\frac{\partial^{3} u}{\partial t \partial x^{2}}(x_{j},t^{n}+\tau_{n}\theta_{4})\\
		&&\ \ \ \ -\frac{h_{n}^{2}}{24}\frac{\partial^{4} u}{\partial x^{4}}(x_{j}+h_{n}\theta_{5},t^{n+1})-\frac{h_{n}^{2}}{24}\frac{\partial^{4} u}{\partial x^{4}}(x_{j}-h_{n}\theta_{6},t^{n+1})-\left(u(x_{j},t^{n})\right)^{p}\\
		&&\ \ \ \ \ \ +\left|\frac{\partial u}{\partial x}(x_{j},t^{n})\right|^{q}+o(\tau_{n})+o(h_{n}^{2}).
	\end{eqnarray*}
	If we put
	\begin{equation*}
	R_{1}=\max_{x,t}\left|\frac{1}{2}\frac{\partial^{2} u}{\partial t^{2}}(x,t)+\frac{\partial^{3} u}{\partial t \partial x^{2}}(x,t)\right| \text{ and } 
	R_{2}=\frac{1}{12}\max_{x,t}\left|\frac{\partial^{4} u}{\partial x^{4}}(x,t)\right|,
	\end{equation*}
	we can deduce that 
	\begin{equation*}
	\max_{1\leq j \leq N_{n}}\epsilon_{j}^{n}\leq C_{1}\tau_{n}+C_{2}h_{n}^{2},
	\end{equation*}
	with $C_{1}\tau_{n}=R_{1}\tau_{n}+o(\tau_{n})$ and $C_{2}h_{n}^{2}=R_{2}h_{n}^{2}+o(h_{n}^{2}).$
	\subsection{Local convergence} 
	Let $e_{j}^{n}=u_{j}^{n}-u(x_{j},t^{n})$ for $j=1,...,m-2.$\\
	\textbf{(A):} Using \eqref{etoi}, \eqref{et} and \ref{starstar} we get
	\begin{eqnarray*}
		&&\frac{u(x_{j},t^{n+1})-u(x_{j},t^{n})}{\tau_{n}}-\frac{u(x_{j+1},t^{n+1})-2u(x_{j},t^{n+1})+u(x_{j-1},t^{n+1})}{h_{n}^{2}}
		-\left(u(x_{j},t^{n})\right)^{p}\\
		&&\ \ \ \ \ \ \  +\left|\frac{u(x_{j+1},t^{n})-u(x_{j-1},t^{n})}{2h_{n}}\right|^{q}\\
		&=&\frac{\tau_{n}}{2}\frac{\partial^{2}u}{\partial t^{2}}(x_{j},t^{n}+\theta_{1}\tau_{n})-\tau_{n}\frac{\partial^{3} u}{\partial t \partial x^{2}}(x_{j},t^{n}+\theta_{6}\tau_{n})\\
		&&\ \ \ \ \ \ \ -\frac{h_{n}^{2}}{24}\left\{\frac{\partial^{4}u}{\partial x^{4}}(x_{j}+\theta_{7}h_{n},t^{n+1})+\frac{\partial^{4}u}{\partial x^{4}}(x_{j}-\theta_{8}h_{n},t^{n+1})\right\}+o(h_{n}^{2}).
	\end{eqnarray*}
	Let 
	\begin{eqnarray*}
		r_{j}^{n}&:=&-\frac{\tau_{n}}{2}\frac{\partial^{2}u}{\partial t^{2}}(x_{j},t^{n}+\theta_{1}\tau_{n})+\tau_{n}\frac{\partial^{3} u}{\partial t \partial x^{2}}(x_{j},t^{n}+\theta_{6}\tau_{n})\\
		&&\ \ \ \ \ +\frac{h_{n}^{2}}{24}\left\{\frac{\partial^{4}u}{\partial x^{4}}(x_{j}+\theta_{7}h_{n},t^{n+1})+\frac{\partial^{4}u}{\partial x^{4}}(x_{j}-\theta_{8}h_{n},t^{n+1})\right\}+o(h_{n}^{2})
	\end{eqnarray*}
	Then 	
	\begin{eqnarray}
	&&\frac{u(x_{j},t^{n+1})-u(x_{j},t^{n})}{\tau_{n}}-\frac{u(x_{j+1},t^{n+1})-2u(x_{j},t^{n+1})+u(x_{j-1},t^{n+1})}{h_{n}^{2}}-\left(u(x_{j},t^{n})\right)^{p} \nonumber \\
	&&\ \ \ +\left|\frac{u(x_{j+1},t^{n})-u(x_{j-1},t^{n})}{2h_{n}}\right|^{q}=-r_{j}^{n}.
	\label{C}
	\end{eqnarray}
	Using \eqref{approchee}, we have 
	\begin{equation}
	\frac{u_{j}^{n+1}-u_{j}^{n}}{\tau_{n}}-\frac{u_{j+1}^{n+1}-2u_{j}^{n+1}+u_{j-1}^{n+1}}{h_{n}^{2}}-(u_{j}^{n})^{p}+\frac{1}{(2h_{n})^{q}}|u_{j+1}^{n}-u_{j-1}^{n}|^{q-1}|u_{j+1}^{n+1}-u_{j-1}^{n+1}|=0.
	\label{D}
	\end{equation}
	From \eqref{C} and \eqref{D}, $e_{j}^{n}$ satisfies
	\begin{eqnarray*}
		&&\frac{e_{j}^{n+1}-e_{j}^{n}}{\tau_{n}}-\frac{e_{j+1}^{n+1}-2e_{j}^{n+1}+e_{j-1}^{n+1}}{h_{n}^{2}}-\left((u_{j}^{n})^{p}-u(x_{j},t^{n})^{p}\right)\\
		&&\ \ \ \ \ +\frac{1}{(2h_{n})^{q}}|u_{j+1}^{n}-u_{j-1}^{n}|^{q-1}|u_{j+1}^{n+1}-u_{j-1}^{n+1}|
		-\left|\frac{u(x_{j+1},t^{n})-u(x_{j-1},t^{n})}{2h_{n}}\right|^{q}\\
		&=& r_{j}^{n}.
	\end{eqnarray*}
	By the mean-value Theorem, for $f(X)=X^{p}$, we get
	\begin{eqnarray*}
		(u_{j}^{n})^{p}-(u(x_{j},t^{n}))^{p}&=&f(u_{j}^{n})-f(u(x_{j},t^{n})\\
		&=&f'(u(x_{j},t^{n})+\theta_{9}e_{j}^{n})e_{j}^{n},
	\end{eqnarray*}
	for some $\theta_{9}\in [0,1]$. Then we obtain
	\begin{eqnarray*}
		&&\frac{e_{j}^{n+1}-e_{j}^{n}}{\tau_{n}}-\frac{e_{j+1}^{n+1}-2e_{j}^{n+1}+e_{j-1}^{n+1}}{h_{n}^{2}}\\
		&=&f'(u(x_{j},t^{n})+\theta_{9}e_{j}^{n})e_{j}^{n}-\frac{1}{(2h_{n})^{q}}|u_{j+1}^{n}-u_{j-1}^{n}|^{q-1}|u_{j+1}^{n+1}-u_{j-1}^{n+1}|\\
		&&\ \ \ \ \ +\left|\frac{u(x_{j+1},t^{n})-u(x_{j-1},t^{n})}{2h_{n}}\right|^{q}+r_{j}^{n}.\\
	\end{eqnarray*}
	Using \eqref{starstar} we get
	\begin{eqnarray*}
		&&\frac{e_{j}^{n+1}-e_{j}^{n}}{\tau_{n}}-\frac{e_{j+1}^{n+1}-2e_{j}^{n+1}+e_{j-1}^{n+1}}{h_{n}^{2}}\\
		&=&f'(u(x_{j},t^{n})+\theta_{5}e_{j}^{n})e_{j}^{n}-\frac{1}{(2h_{n})^{q}}|u_{j+1}^{n}-u_{j-1}^{n}|^{q-1}|u_{j+1}^{n+1}-u_{j-1}^{n+1}|+\left|\frac{\partial u}{\partial x}(x_{j},t^{n})\right|^{q}+r_{1j}^{n}.
	\end{eqnarray*}
	with $r_{1j}^{n}=r_{j}^{n}+o(h_{n}^{2}).$\\
	Let:
	\begin{align*}
	E^{n}&=\max_{1\leq j \leq m-2}\left|e_{j}^{n}\right|, & U=&\max_{x,t}\left|u(x,t)\right|,	& V&=\max_{x,t}\left|\frac{\partial u}{\partial x}(x,t)\right|, \\ W&=\frac{2}{3}\max_{x,t}\left|\frac{\partial^{3} u}{\partial x^{3}}(x,t)\right|,
	& K&=f'(U+1),
	\end{align*}
	and
	\begin{equation*}
	R=\frac{\lambda_{n}}{2}\max_{x,t}\left|\frac{\partial^{2}u}{\partial t^{2}}(x,t)\right|+\lambda_{n}\max_{x,t}\left|\frac{\partial^{2}u}{\partial x^{2}}(x,t)\right|+\frac{1}{12}\max_{x,t}\left|\frac{\partial^{4}u}{\partial x^{4}}(x,t)\right|+o(1)+o(\lambda_{n}).
	\end{equation*}
	But from \eqref{E} we have
	\begin{eqnarray}
	\nonumber &&\left|\frac{\partial u}{\partial x}(x,t)-\frac{u_{j+1}^{n}-u_{j-1}^{n}}{2h_{n}}\right|\\
	\nonumber &=&\left|\frac{u(x_{j+1},t^{n})-u(x_{j-1},t^{n})}{2h_{n}}-\frac{h_{n}^{2}}{3}\frac{\partial^{3}u}{\partial x^{3}}(x_{j}+\theta h_{n},t^{n})+\frac{h_{n}^{2}}{3}\frac{\partial^{3}u}{\partial x^{3}}(x_{j}-\theta h_{n},t^{n}) -\frac{u_{j+1}^{n}-u_{j-1}^{n}}{2h_{n}}\right|\\
	\nonumber &=&\left|\frac{e_{j-1}^{n}-e_{j+1}^{n}}{2h_{n}}-\frac{h_{n}^{2}}{3}\frac{\partial^{3}u}{\partial x^{3}}(x_{j}+\theta h_{n},t^{n})-\frac{h_{n}^{2}}{3}\frac{\partial^{3}u}{\partial x^{3}}(x_{j}-\theta h_{n},t^{n})\right|\\
	&\leq& \frac{E^{n}}{h_{n}}+h_{n}^{2}W.
	\label{F}
	\end{eqnarray}
	Then by \eqref{F} and the mean value theorem, for $g(X)=\left|X\right|^{q}$, we get
	\begin{eqnarray}
	\nonumber &&\left|\left|\frac{\partial u}{\partial x}(x_{j},t^{n})\right|^{q}-\left|\frac{u_{j+1}^{n}-u_{j-1}^{n}}{2h_{n}}\right|^{q}\right|\\
	&\leq&qg'\left(\frac{\partial u}{\partial x}(x_{j},t^{n})+\theta\left(\frac{E^{n}}{h_{n}}+h_{n}^{2}W\right)\right)\left(\frac{E^{n}}{h_{n}}+h_{n}^{2}W\right).
	\label{H}
	\end{eqnarray}
	In the other hand, for $1\leq j\leq m-2$ we have,
	\begin{eqnarray}
	\nonumber &&\left|\left|\frac{u_{j+1}^{n}-u_{j-1}^{n}}{2h_{n}}\right|^{q}-\left|\frac{u_{j+1}^{n}-u_{j-1}^{n}}{2h_{n}}\right|^{q-1}\left|\frac{u_{j+1}^{n+1}-u_{j-1}^{n+1}}{2h_{n}}\right|\right|\\
	\nonumber &=&\left|\frac{u_{j+1}^{n}-u_{j-1}^{n}}{2h_{n}}\right|^{q-1}\left(\frac{u_{j+1}^{n}-u_{j-1}^{n}}{2h_{n}}-\frac{u_{j+1}^{n+1}-u_{j-1}^{n+1}}{2h_{n}}\right)\\
	\nonumber &\leq&\left|\frac{u_{j+1}^{n}-u_{j-1}^{n}}{2h_{n}}\right|^{q-1}\left(\frac{e_{j+1}^{n}-e_{j-1}^{n}}{2h_{n}}-\frac{e_{j+1}^{n+1}-e_{j-1}^{n+1}}{2h_{n}}+o(\tau_{n})+o(h_{n}^{2})\right)\\
	\nonumber &\leq&\left|\frac{u_{j+1}^{n}-u_{j-1}^{n}}{2h_{n}}\right|^{q-1}\left(\frac{e_{j+1}^{n}-e_{j-1}^{n}}{2h_{n}}\right)+\left|\frac{u_{j+1}^{n}-u_{j-1}^{n}}{2h_{n}}\right|^{q-1}\left(\frac{e_{j+1}^{n+1}-e_{j-1}^{n+1}}{2h_{n}}\right)\\
	&&\ \ \ \ \ \ +\left(o(\tau_{n})+o(h_{n}^{2})\right)\left|\frac{u_{j+1}^{n}-u_{j-1}^{n}}{2h_{n}}\right|^{q-1}.
	\label{G}
	\end{eqnarray}
	Then from \eqref{G} and \eqref{H} we get
	\begin{eqnarray*}
		&&\frac{e_{j}^{n+1}-e_{j}^{n}}{\tau_{n}}-\frac{e_{j+1}^{n+1}-2e_{j}^{n+1}+e_{j-1}^{n+1}}{h_{n}^{2}}\\
		&\leq&f'(u(x_{j},t^{n})+\theta_{9}e_{j}^{n})e_{j}^{n}+r_{1j}^{n}+\left(o(\tau_{n})+o(h_{n}^{2})\right)\left|\frac{u_{j+1}^{n}-u_{j-1}^{n}}{2h_{n}}\right|^{q-1}\\
		&+&qg'\left(\frac{\partial u}{\partial x}(x_{j},t^{n})+\theta\left(\frac{E^{n}}{h_{n}}+h_{n}^{2}W\right)\right)\left(\frac{E^{n}}{h_{n}}+h_{n}^{2}W\right)+\left|\frac{u_{j+1}^{n}-u_{j-1}^{n}}{2h_{n}}\right|^{q-1}\left(\frac{e_{j+1}^{n}-e_{j-1}^{n}}{2h_{n}}\right)\\
		&&\ \ \  \ +\left|\frac{u_{j+1}^{n}-u_{j-1}^{n}}{2h_{n}}\right|^{q-1}\left(\frac{e_{j+1}^{n+1}-e_{j-1}^{n+1}}{2h_{n}}\right).
	\end{eqnarray*}
	Let $M:=\left\|U^{n}\right\|_{\infty}=u_{m}^{n}.$
	Finally we obtain
	\begin{eqnarray*}
		\frac{E^{n+1}-E^{n}}{\tau_{n}}&\leq& KE^{n}+h_{n}^{2}R+\left(o(\tau_{n})+o(h_{n}^{2})\right)\left(\frac{M}{h_{n}}\right)^{q-1}\\
		&+&qg'\left(V+\theta\left(\frac{E^{n}}{h_{n}}+h_{n}^{2}W\right)\right)\left(\frac{E^{n}}{h_{n}}+h_{n}^{2}W\right)\\
		&=&E^{n}\left(K+\frac{q}{h_{n}}g'\left(V+\theta\left(\frac{E^{n}}{h_{n}}+h_{n}^{2}W\right)\right)\right)\\
		&+&h_{n}^{2}\left(R+\left(o(\lambda_{n})+o(1)\right)\left(\frac{M}{h_{n}}\right)^{q-1}+Wqg'\left(V+\theta\left(\frac{E^{n}}{h_{n}}+h_{n}^{2}W\right)\right)\right)\\
		&=&\frac{E^{n}}{h_{n}^{q}}\left(h_{n}^{q}K+qg'\left(h_{n}V+\theta\left(E^{n}+h_{n}^{3}W\right)\right)\right)\\
		&+&h_{n}^{3-q}\left(h_{n}^{q-1}R+\left(o(\lambda_{n})+o(1)\right)M^{q-1}+Wqg'\left(h_{n}V+\theta\left(E^{n}+h_{n}^{3}W\right)\right)\right).
	\end{eqnarray*}
	Let 
	\begin{eqnarray*}
		B&=&h_{n}^{q}K+qg'\left(h_{n}V+\theta\left(E^{n}+h_{n}^{3}W\right)\right)\\
		C&=&h_{n}^{q-1}R+\left(o(\lambda_{n})+o(1)\right)M^{q-1}+Wqg'\left(h_{n}V+\theta\left(E^{n}+h_{n}^{3}W\right)\right).
	\end{eqnarray*}
	Then,
	\begin{eqnarray*}
		E^{n+1}&\leq&\left(1+\tau_{n}\frac{B}{h_{n}^{q}}\right)E^{n}+\tau_{n}h_{n}^{3-q}C\\
		&\leq&\left(1+\tau_{n}NB\right)E^{n}+\tau_{n}h_{n}^{3-q}C\\
		&\leq&\exp(NBT)h_{n}^{3-q}CT.
	\end{eqnarray*}
	With $N$ is constant such that
	\begin{equation*}
	\text{For } \  t_{n}<T \text{  and } h_{n}=\left(2M^{-q+1}\right)^{\frac{1}{2-q}} \text{  we have: } \ \ \dfrac{1}{h_{n}^{q}}=\dfrac{M^{\frac{q(q-1)}{2-q}}}{2^{\frac{q}{2-q}}}:=N,
	\end{equation*}
	which is bounded by Theorem 2.2.
	Then we get
	\begin{equation*}
	\max_{1\leq j \leq m-2}\left|u_{j}^{n}-u(x_{j},t^{n})\right|\leq C_{0}(T)h^{3-q}.\\
	\end{equation*}
	Now, we will prove the last part of the lemma.\\
	\textbf{(B):} We do the same thing for $p>1$ and $q=1$, we get for $j=1,...,m-1$
	\begin{eqnarray*}
		&&\frac{e_{j}^{n+1}-e_{j}^{n}}{\tau_{n}}-\frac{e_{j+1}^{n+1}-2e_{j}^{n+1}+e_{j-1}^{n+1}}{h_{n}^{2}}\\
		&=&f'(u(x_{j},t^{n})+\theta_{9}e_{j}^{n})e_{j}^{n}-\frac{u_{j+1}^{n+1}-u_{j-1}^{n+1}}{2h_{n}}+\frac{u(x_{j+1},t^{n})-u(x_{j-1},t^{n})}{2h_{n}}+r_{j}^{n}\\
		&=&f'(u(x_{j},t^{n})+\theta_{9}e_{j}^{n})e_{j}^{n}+\frac{e_{j-1}^{n+1}-e_{j+1}^{n+1}}{2h_{n}}+r_{j}^{n}.
	\end{eqnarray*}
	And then
	\begin{eqnarray*}
		&&\frac{E^{n+1}-E^{n}}{\tau_{n}}\leq KE^{n}+h_{n}^{2}R.\\
		&\Rightarrow& E^{n+1}\leq \tau_{n}KE^{n}+\tau_{n}h_{n}^{2}R.\\
		&\Rightarrow& E^{n+1}\leq \exp(KT)h_{n}^{2}RT.
	\end{eqnarray*}
	And finally we obtain
	\begin{equation*}
	\max_{1\leq j\leq m-1}\left|u_{j}^{n}-u(x_{j},t^{n})\right|\leq C_{1}(T)h^{2}.
	\end{equation*}
	\section{Approximation of the blowing up time}
	In this section, we give an idea about the numerical blow-up time. First of all we recall a result of Souplet and Weissler \cite{soupletweissler}
	\begin{th1}
		Let $\psi \in W^{1,s}_{0}(\Omega)$, ($s$ large enough), with $\psi\geq 0$ and $\psi \neq 0.$
		\begin{enumerate}
			\item There exists some $\lambda_{0}=\lambda_{0}(\psi)>0$ such that for all $ \lambda>\lambda_{0}$, the solution of \eqref{exacte} with initial data $\phi=\lambda \psi$ blows up in finite time in $W^{1,s}$ norm.
			\item There is some $C>0$ such that 
			\begin{equation*}
			T^{*}(\lambda \psi)\leq \frac{C}{(\lambda \left| \psi \right|_{\infty})^{p-1}},\ \ \ \lambda\rightarrow \infty.
			\end{equation*}
			\item \begin{equation*}
			T^{*}(\lambda \psi)\geq \frac{1}{(p-1)(\lambda \left| \psi \right|_{\infty})^{p-1}}.
			\end{equation*}
		\end{enumerate}
	\end{th1}
	We define now 
	\begin{equation}
	T_{num}^{*}:=\sum_{n\geq 0}{\tau_{n}}
	\label{timeblowup}
	\end{equation}
	and call it the numerical blow-up time.
	In \cite{hani}, we have proved that 
	\begin{equation*}
	u_{m}^{n}\geq \left(\dfrac{1+\tau}{1+\tau 2^{\frac{-q}{2-q}}\left(u_{m}^{0}\right)^{\frac{-2p+q(1+p)}{2-q}}}\right)^{n}u_{m}^{0}.  
	\end{equation*}
	which implies that
	\begin{equation}
	\dfrac{1}{(u_{m}^{n})^{p-1}}\leq \dfrac{1}{\left(\dfrac{1+\tau}{1+\tau 2^{\frac{-q}{2-q}}(u_{m}^{0})^{\frac{-2p+q(1+p)}{2-q}}}\right)^{n(p-1)}}(u_{m}^{0})^{-p+1}.
	\label{time}
	\end{equation}
	Using \eqref{timeblowup} and \eqref{time} we get
	\begin{eqnarray}
	T_{num}^{*}&=&\tau \sum_{n\geq 0}\dfrac{1}{(u_{m}^{n})^{p-1}} \nonumber \\
	&\leq& \dfrac{\tau}{(u_{m}^{0})^{p-1}}\sum_{n\geq 0}\left(\dfrac{1}{\left(\frac{1+\tau}{1+\tau 2^{\frac{-q}{2-q}}\left(u_{m}^{0}\right)^{\frac{-2p+q(1+p)}{2-q}}}\right)^{p-1}}\right)^{n} \nonumber \\
	&=&\dfrac{\tau}{(u_{m}^{0})^{p-1}}\sum_{n\geq 0}\left(\left(\dfrac{1+\tau 2^{\frac{-q}{2-q}}\left(u_{m}^{0}\right)^{\frac{-2p+q(1+p)}{2-q}}}{1+\tau}\right)^{p-1}\right)^{n} \nonumber \\
	&=&\dfrac{\tau}{(u_{m}^{0})^{p-1}}\dfrac{1}{1-\left(\dfrac{1+\tau 2^{\frac{-q}{2-q}}\left(u_{m}^{0}\right)^{\frac{-2p+q(1+p)}{2-q}}}{1+\tau}\right)^{p-1}}:=T^{**}
	\label{tstar}
	\end{eqnarray}
	\section{Numerical simulations}
	In this section, we present some numerical simulations that illustrate our results. In figure 1, we take $p=4>2$ and $q=1.3<\frac{2(p-1)}{p}$, one can see that the solution is bounded in $x_{m-1}$. Then we take $p=2$ and $q=1$, it is clear from figure 2 that the solution blows up in $x_{m-1},$ and from figure 3, we can see that the solution is bounded in $x_{m-2}.$\\
	Concerning the approximation of the blowing up time, if we take the initial data $u_{0}(x)=\lambda \sin(\frac{\pi}{2}(x+1))$, with $\lambda>0$ then $\left\|u_{0}\right\|_{\infty}=\lambda$. Theoretically we know that 
	\begin{equation*}
	T^{*}\geq \dfrac{1}{(p-1)\left\|u_{0}\right\|_{\infty}^{p-1}}.
	\end{equation*}
	Let $g(\lambda)=\dfrac{1}{(p-1)\lambda^{p-1}}$ and $p=3$. In the next table, and for some values of $\lambda$ we can see that $T^{*}_{num}\geq g(\lambda)$ which is compatible with the theoretical result, this is illustrated in figure 4.
	 Also, using \eqref{tstar} and for $\lambda=10^{3}$ we have $$T^{*}_{num}\approx 5.067.10^{-7}\leq T^{**}=5.075.10^{-7}.$$ 
	\begin{table}[H]
		\centering
		\begin{tabular}{|*{6}{c|}}
			\hline
			$\lambda$& $10$&$10^{2}$&$10^{3}$&$10^{4}$&$10^{5}$ \rule[-7pt]{0pt}{20pt}\\ \hline
			$g(\lambda)$&$5.10^{-3}$&$5.10^{-5}$&$5.10^{-7}$&$5.10^{-9}$&$5.10^{-11}$ \rule[-7pt]{0pt}{20pt}\\ \hline
			$T^{*}_{num}$&$5.177.10^{-3}$&$5.068.10^{-5}$&$5.067.10^{-7}$&$5.075.10^{-9}$&$5.058.10^{-11}$ \rule[-7pt]{0pt}{20pt}\\ \hline
		\end{tabular}
		\vspace{0.5cm}
			\caption{Comparison of the function $g(\lambda)$ with the numerical blow up time $T^{*}_{num}$.}
			\label{tab:1}       %
	\end{table}	
	\begin{figure}[H]
		\begin{minipage}[b]{0.40\linewidth}
			\centering
			\includegraphics[width=3in,height=2in]{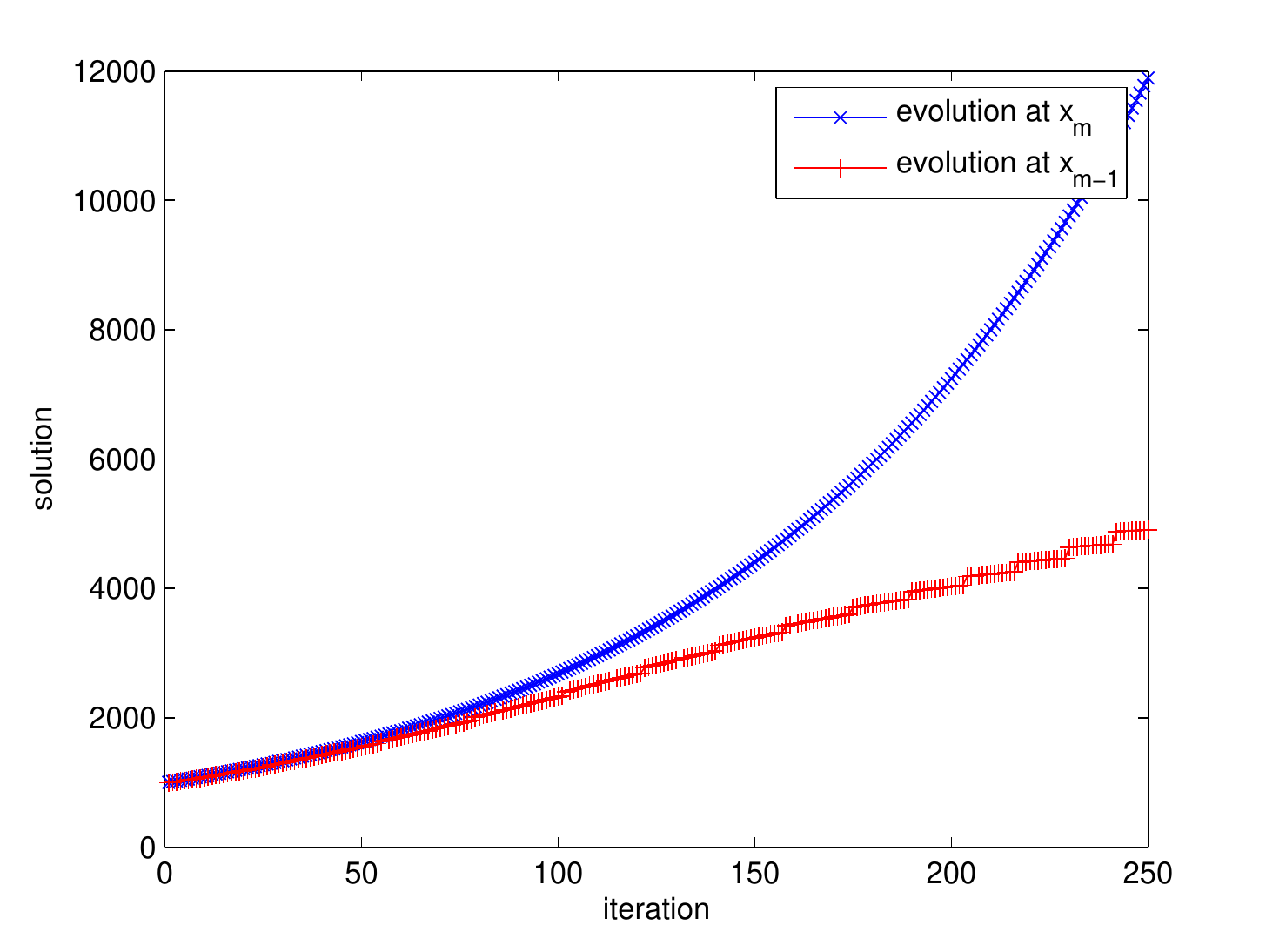}
			\caption{Evolution of the numerical solution at $x_{m}$ and $x_{m-1}$ for $p=4$ and $q=1.3$}
		\end{minipage}\hfill
		\begin{minipage}[b]{0.40\linewidth}
			\centering
			\includegraphics[width=3in,height=2in]{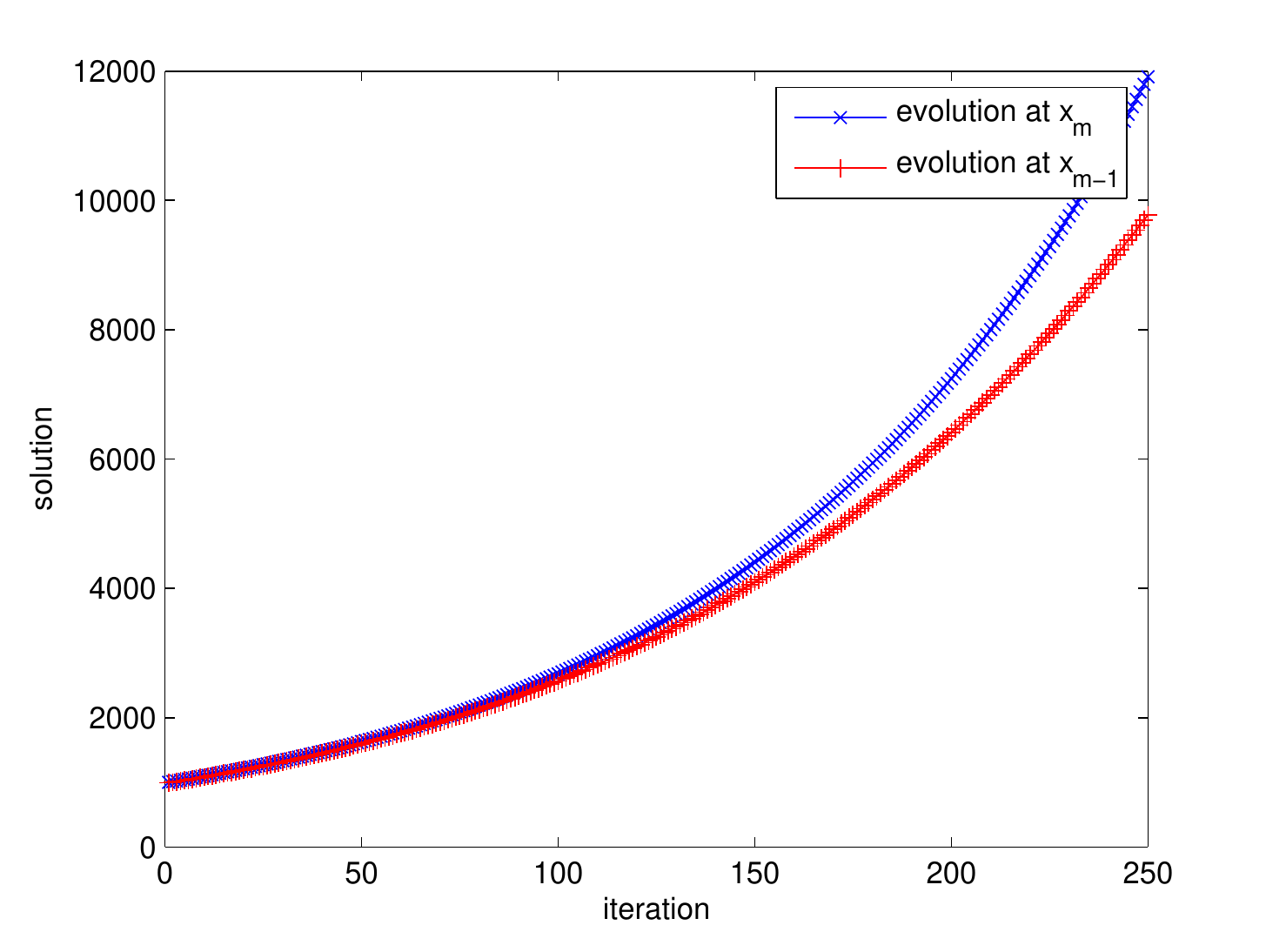}
			\caption{Evolution of the numerical solution at $x_{m}$ and $x_{m-1}$ for $p=2$ and $q=1.$}
		\end{minipage}
	\end{figure}
	\begin{figure}[H]
		\begin{minipage}[b]{0.40\linewidth}
			\centering
			\includegraphics[width=3in,height=2in]{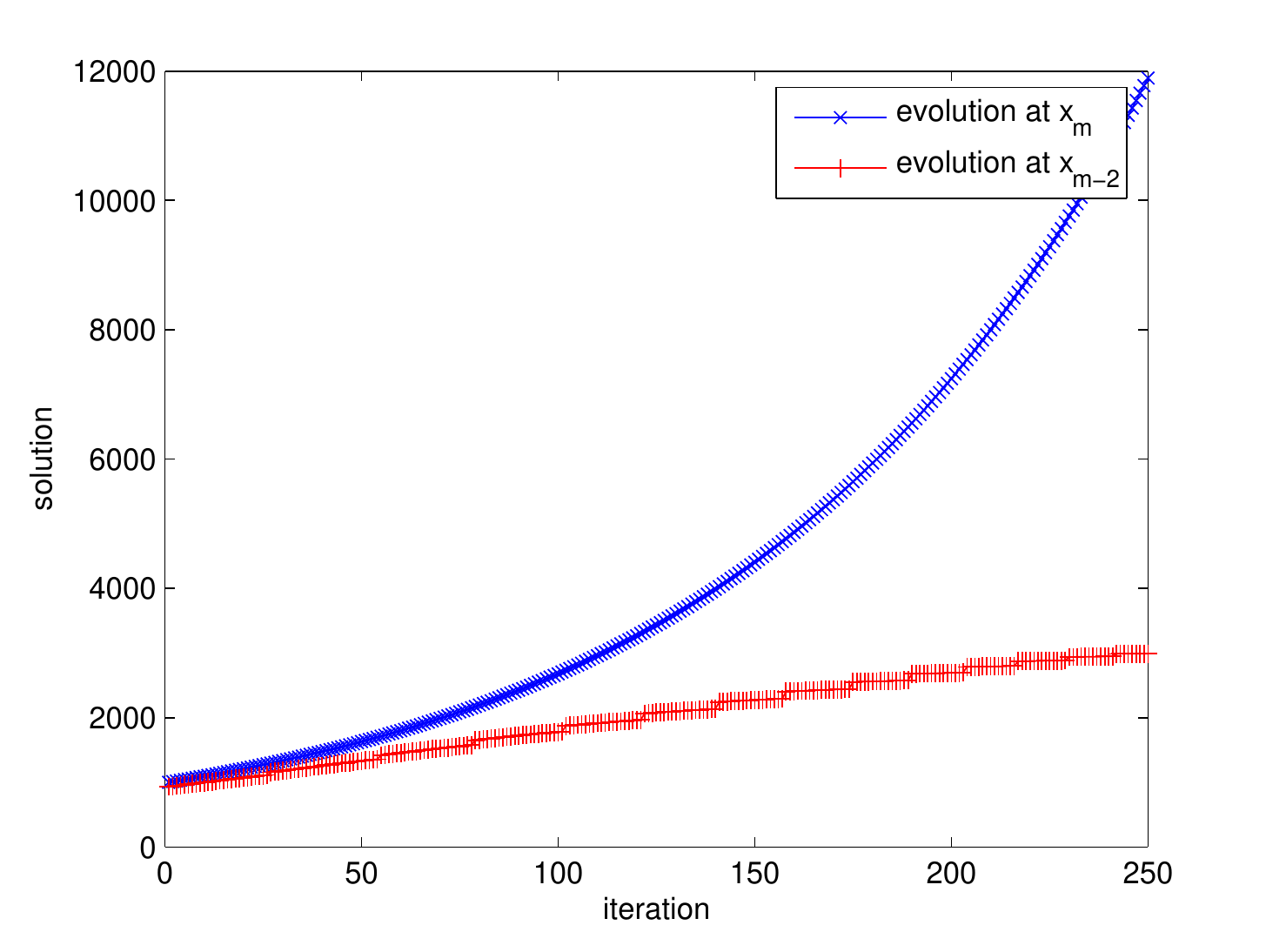}
			\caption{Evolution of the numerical solution at $x_{m}$ and $x_{m-2}$ for $p=2$ and $q=1.$}
		\end{minipage}\hfill
		\begin{minipage}[b]{0.40\linewidth}
			\centering
			\includegraphics[width=3in,height=2in]{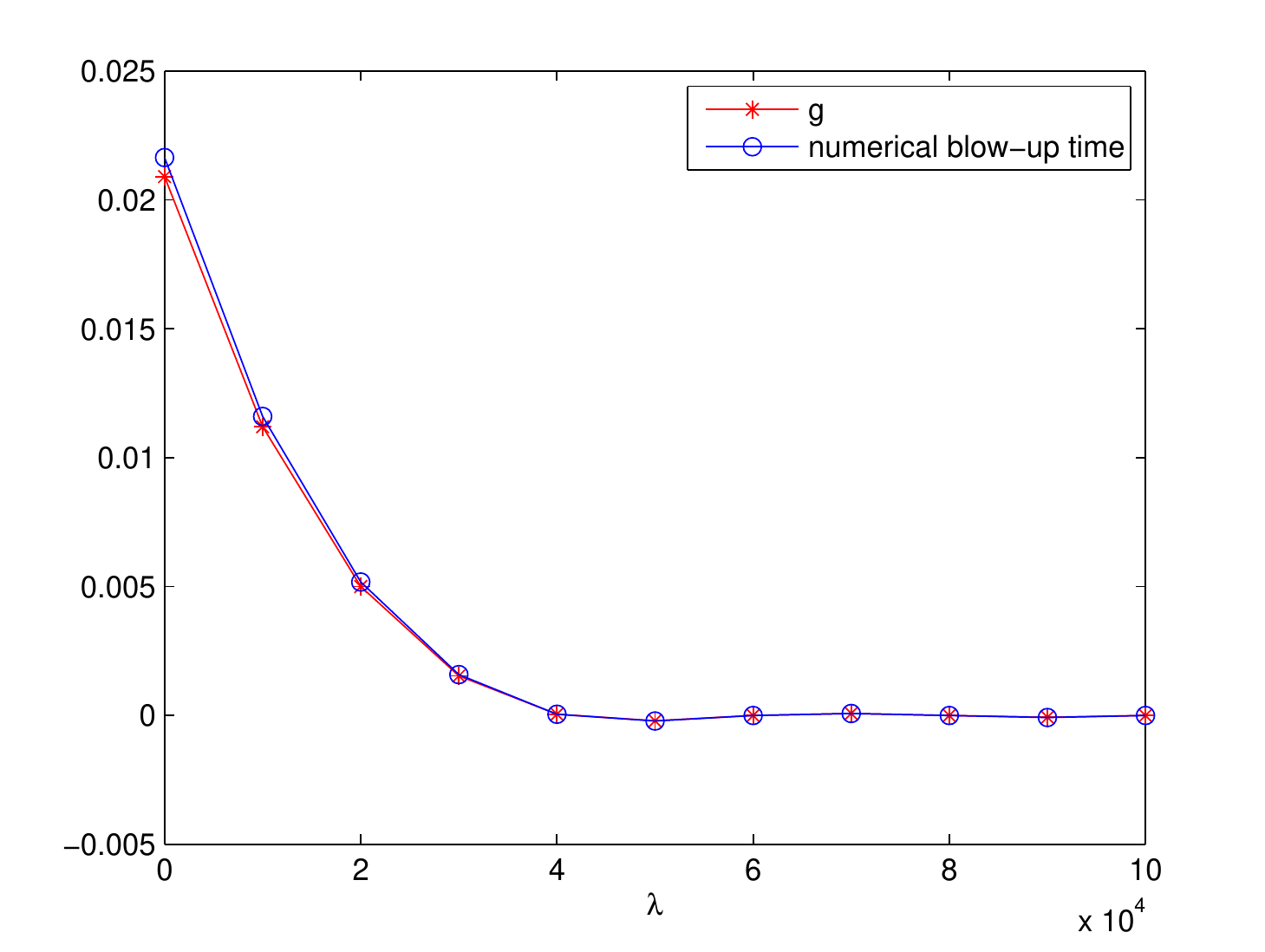}
			\caption{Graphics of $g(\lambda)$ and approximation of the numerical blow-up time for $p=3.$}
		\end{minipage}
	\end{figure}
	\section{Conclusion }
	\noindent We have showed that when $p=2$ and $q=1$, the finite difference solution blows up at more than one point and that when $p>2$ and $q<\dfrac{2(p-1)}{p}$, the only numerical blow up point is the mid-point $x=0.$
	This is an interesting phenomena in view of the fact that the solution of the corresponding PDE blows up only at one point $x=0$ for any $p>1$ and $1\leq q\leq \dfrac{2p}{p+1}.$
	Remark that for $1<p<2$ and $\dfrac{2(p-1)}{p}\leq q< \dfrac{2p}{p+1},$ we have no idea about the boundedness of $u_{m-1}^{n}$ and $u_{m-2}^{n}.$
	\vspace{-2cm}
	\begin{figure}[H]
		\centering
		\includegraphics[width=10cm,height=10cm]{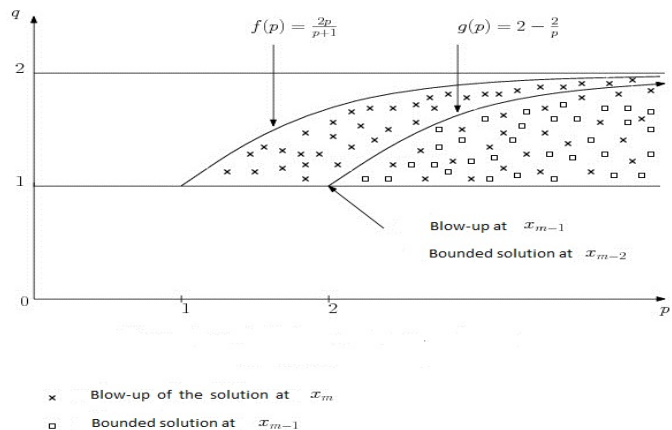}
		\vspace{-2cm}
		\caption{Graphics of the asymptotic behaviours of the solution near the blowing-up point.}
	\end{figure}

\end{document}